\documentclass[11pt,reqno]{amsart}
\usepackage[a4paper]{geometry}
\usepackage{amsfonts,amsmath,amssymb}
\usepackage{mathabx}
\usepackage{mathrsfs}
\usepackage[latin1]{inputenc}
\usepackage[usenames,dvipsnames]{pstricks}
\newtheorem{theorem}{Theorem}[section]

\newtheorem{corollary}[theorem]{Corollary}

\newtheorem{definition}[theorem]{Definition}

\newtheorem{lemma}[theorem]{Lemma}

\newcommand{\R}{\mathbb{R}}

\newcommand{\M}{\Omega}
\newcommand{\ek}{\phi}

\newcommand{\g}{\mathrm{g}}

\newcommand{\dd}{\mathrm{d}}
\newcommand{\dv}{\,\mathrm{dv}^{n}}

\newcommand{\dvv}{\,\mathrm{dv}^{2n-1}}
\newcommand{\ds}{\,\mathrm{d\sigma}^{n-1}}

\newcommand{\dss}{\,\mathrm{d\sigma}^{2n-2}}
\newcommand{\I}{\mathcal{I}}

\newcommand{\p}{\partial}

\newcommand{\norm}[1]{\|#1\|}
\newcommand{\abs}[1]{|#1|}
\newcommand{\set}[1]{\left\{#1\right\}}
\newcommand{\para}[1]{\left(#1\right)}

\newcommand{\seq}[1]{\left<#1\right>}

\newcommand{\To}{\longrightarrow}

\newcommand{\nr}{\aleph}
\usepackage{graphicx}
\newcommand{\dive}{\textrm{div}}

\newcommand{\bea}{\begin{eqnarray}}
\newcommand{\eea}{\end{eqnarray}}
\newcommand{\beas}{\begin{eqnarray*}}
\newcommand{\eeas}{\end{eqnarray*}}
\newcommand{\bel}{\begin{equation} \label}
\newcommand{\ee}{\end{equation}}

\usepackage{times}
\allowdisplaybreaks
\begin{document}

\title[Recovery of a metric tensor from the partial hyperbolic D-to-N map]{Stable recovery of a metric tensor from the partial hyperbolic Dirichlet to Neumann map}
\author[M. Bellassoued]{Mourad~Bellassoued}
\address{University of Tunis El Manar, National Engineering School of Tunis, ENIT-LAMSIN, B.P. 37, 1002 Tunis, Tunisia}
\email{mourad.bellassoued@enit.utm.tn}
\date{\today}
\begin{abstract}
In this paper we consider the inverse problem of determining on a compact Riemannian manifold
the metric tensor in the wave equation
 with Dirichlet  data from measured Neumann boundary observations. This information is enclosed in the dynamical
Dirichlet-to-Neumann map associated to the wave equation. We prove in dimension $n\geq 2$ that the
knowledge of the partial Dirichlet-to-Neumann map for the wave equation uniquely determines
the metric tensor and we establish logarithm-type stability.
\end{abstract}
\maketitle
\section{Introduction and main results}
\subsection{Statement of the problem}
Let $\M$ be a bounded and connected domain of $\R^n$, $n\geq 2$, with smooth boundary $\Gamma=\p\M$. Assume that $\g=(\g_{jk})$ is given Riemannian metric in $\M$, with symmetric and smooth coefficients $\g_{jk}\in\mathcal{C}^k(\M)$, $k\geq 2$, and 
\begin{equation}\label{1.1}
\g(x)=\sum_{j,k=1}^n\g_{jk}(x)dx_j\otimes dx_k.
\end{equation}
We let $\Delta_\g$ denote the Laplace-Beltrami operator on $\M$, given by
\begin{equation}\label{1.2}
\Delta_\g=\frac{1}{\sqrt{\abs{\g}}}\sum_{j,k=1}^n\frac{\p}{\p
x_j}\para{\sqrt{\abs{\g}}\,\g^{jk}\frac{\p}{\p x_k}}.
\end{equation}
Here $(\g^{jk})$ is the inverse of the metric $\g$ and $\abs{\g}=\det(\g_{jk})$. A summary of the main Riemannian geometric notions needed in this paper is provided in Section 2. Given $T>0$, we denote $Q=(0,T)\times\M$ and $\Sigma=(0,T)\times \Gamma$. We consider the
following initial boundary value problem for the wave equation,
\begin{equation}\label{1.3}
\left\{
\begin{array}{llll}
\para{\partial^2_t-\Delta_\g}u=0  & \textrm{in }\; Q,\cr
u(0,\cdot)=\p_tu(0,\cdot)=0 & \textrm{in }\; \M ,\cr
u=f & \textrm{on } \; \Sigma.
\end{array}
\right.
\end{equation}
The purpose of this paper is to study the inverse problem of determining the metric tensor $\g$ appearing in the wave equation (\ref{1.3}). 
Before going to the main context of the paper, let us briefly mention about the well-posedness of the forward problem (\ref{1.3}). Following \cite{BF}, we define the space $H^{1,1}_0(\Sigma)$ by
\begin{equation}\label{1.4}
H^{1,1}_0(\Sigma)=\set{f\in L^2(0,T;H^2(\Gamma))\cap H^1(0,T;L^2(\Gamma)),\,\textrm{and}\,\, f(0,x)=0,\,\textrm{for}\,x\in\Gamma}.
\end{equation}
As shown in \cite{BF} that for $f\in H^{1,1}_0(\Sigma)$, system (\ref{1.3}) admits a unique solution
\begin{equation}\label{1.5}
u\in \mathcal{C}^1([0,T];L^2(\M))\cap \mathcal{C}([0,T];H^1(\M)).
\end{equation}
Furthermore, there is a constant $C>0$ such that
\begin{equation}\label{1.6}
\norm{\p_\nu u}_{L^2(\Sigma)}\leq C\norm{f}_{H^{1,1}(\Sigma)}.
\end{equation}
The normal derivative is given by
\begin{equation}\label{1.7}
\p_\nu u:=\seq{\nabla u,\nu}=\sum_{j,k=1}^n\g^{jk}\nu_j\frac{\p u}{\p x^k}
\end{equation}
where $\nu$ is the unit outward vector field to $\Gamma$. In particular the following operator, usually called the Dirichlet to Neumann map
\begin{equation}\label{1.8}
\begin{array}{ccc}
\Lambda_{\g}: H_0^{1,1}(\Sigma)&\longrightarrow& L^{2}(\Sigma)\\
f &\longmapsto &\Lambda_{\g}(f):=\p_{\nu}u,
\end{array}
\end{equation}
is bounded from $H_0^{1,1}(\Sigma)$ into $L^{2}(\Sigma)$.
\smallskip

It's known that the problem of finding a Riemannian metric is interesting for theory and applications. For example, the problem of finding a metric from the corresponding Dirichlet to Neumann map  is the basic question of the investigation of inverse problems for partial differential equations and there are many applications of this problem, particularly, in geophysics in connection with the study of distribution of the velocities of propagation of elastic waves inside the terrestrial globe. ln the study of the problem of restoring a Riemannian metric a new type problem of integral geometry appears (geodesic ray transform). It is known that the integral geometry problem is the mathematical base of tomography. This problem has many applications in various fields: the problem of the forecasting of earthquakes, diagnostics of plasma, problem of photometry, fiber optics and etc. The connection of the problem of integral geometry for differential forms with the inverse problems for equations of the hyperbolic type and kinetic equations are described in the works \cite{Amirov1}, \cite{[Sh]}. 
\smallskip

An inverse problem in this setting is whether one can reconstruct the Riemannian metric $\g$ in $\M$ by the knowledge of $\Lambda_\g$. There is a well known non-uniqueness for this inverse problem. For it we let $\psi:\bar{\M}\to \bar{\M}$ be a diffeomorphism with $\psi=\textrm{id}$ on $\Gamma$, then $\Lambda_{\psi^*\g}=\Lambda_\g$, where $\psi^*\g$ denote the pull back of the metric $\g$ by the diffeomorphism $\psi$. 
$$
\psi^*\g={}^t\psi'\cdot (\g\circ\psi)\cdot \psi'
$$
here $\psi'$ denotes the (matrix) differential of the smooth function $\psi$, ${}^t\psi'$ its transpose, and the $\cdot$ represents multiplication of matrices. An interesting problem is to see if this is the only obstruction for uniqueness, i.e., if $\g_1$ and $\g_2$ be two Riemannian metrics on $\M$ such that $\Lambda_{\g_1}=\Lambda_{\g_2}$, then there exists a diffeomorphism $\psi:\bar{\M}\to \bar{\M}$ with $\psi=\textrm{id}$ on $\Gamma$ such that $\Lambda_{\psi^*\g}=\Lambda_\g$.  A Riemannian manifold $(\M,\g)$ is said to be uniquely determined up to isometry,  if this is the only obstruction to unique identifiability of the metric.  The Dirichlet to Neumann map can be related to the geodesic distance $d_\g$ which measures the travel times of geodesics joining points of the
boundary. In the case that both $\g_1$ and $\g_2$ are conformal to a \textit{simple} metric $\g$ (i.e., $\g_k=c_k\g$), then this problem is known in seismology as the inverse kinematic problem. In this case, it has been proven, by Bellassoued and Dos Santos Ferriera \cite{BF}, under further restrictions on the metrics that if $\Lambda_{c\g} = \Lambda_{\g}$, then $c=1$. In this case the diffeomorphism $\psi$ must be the identity.
The questions naturally arise:
\begin{itemize}
\item[(a)] are there other types of the non-uniqueness of solution of this inverse problem?
\item[(b)] when is a metric tensor determined by its Dirichlet to Neumann map up to isometry, identical on?
\item[(c)]  for what classes of metrics the Dirichlet to Neumann map determines a metric uniquely ($\varphi$ must be the identity)?
\end{itemize}
Inverse problems with partial boundary data are encountered in mathematical physics and in various applications. For example in medical imaging and in the geophysical imaging of the Earth, measurements can usually be done only for a part of the boundary. In this paper, we will analyze the problem of finding in the wave equation (\ref{1.3}) a Riemannian metric $\g$ in the case when some of the Dirichlet- to Neumann map is known only in arbitrary relatively open part $\Gamma^\natural\subset\Gamma$. So that we consider non-empty relatively open subset $\Gamma^\natural$ of $\Gamma$.  Then we set $\Sigma^\natural=(0,T)\times \Gamma^\natural$ and we introduce the \textit{partial} Dirichlet to Neumann map
\begin{equation}\label{1.9}
\begin{array}{ccc}
\Lambda_{\g}^\natural: H_0^{1,1}(\Sigma)&\longrightarrow& L^{2}(\Sigma^\natural)\\
f &\longmapsto &\Lambda_{\g}^\natural(f):=\p_{\nu}u_{|\Sigma^\natural},
\end{array}
\end{equation}
we observe that since $\Lambda_\g$ is bounded, $\Lambda_\g^\natural$ is also bounded from $H_0^{1,1}(\Sigma)$ into $L^{2}(\Sigma^\natural)$.
\smallskip

The main achievement of this paper is that the Neumann data used in this stability estimate can be measured on the sub-part $\Gamma^\natural$ of the whole boundary, where we recall that $\Gamma^\natural$ an arbitrary non-empty relative open subset of $\Gamma$. Here questions connected with the uniqueness and stability estimates of a solution to the problem in question will be discussed.

\subsection{Previous literature}
The recovery of coefficients found in hyperbolic equations is a topic that has drawn significant interest. Several authors have treated the determination of coefficients from the Dirichlet to Neumann map analogous to $\Lambda_\g$ above.  Unique determination of the metric goes back to Belishev and Kurylev \cite{BK} using the boundary control method and involves works of Katchlov, Kurylev and Lassas \cite{KKL}, Kurylev and Lassas \cite{KL}, Lassas and Oksanen  \cite{LO} and Anderson, Katchalov, Kurylev, Lassas and Taylor \cite{AKKLT}. In fact, Katchalov, Kurylev, Lassas and Mandache proved that the determination of the metric from the Dirichlet-to-Neumann map was equivalent for the wave and Schr\"odinger equations (as well as other related inverse problems) in \cite{KKLM}. The importance of control theory for inverse problems was first understood by Belishev \cite{Beli}. He used control theory to develop the first variant of the control (BC) method. This method gives an efficient way to reconstruct a Riemannian manifold via its response operator (dynamical Dirichlet-to-Neumann map) or spectral data (a spectrum of the Beltrami-Laplace operator and traces of normal derivatives of the eigenfunctions), themselves, whereas the coefficients on these manifolds are recovered automatically.
\smallskip

The problem of establishing stability estimates in determining the metric was studied by Stefanov and Uhlmann in \cite{StU2,StU3} for metrics close to Euclidean and generic simple metrics. H\"older type of conditional stability estimate was proven in \cite{StU3} for metrics close enough to the Euclidean one in three dimensions. In \cite{StU2}, the authors prove H\"older type of stability near generic simple metrics. A Riemannian manifold $(\M,\g)$ is simple, if $\M$ is simply connected, $\p\M$ is strictly convex and any two point in $\M$ can be connected by a single minimizing geodesic depending smoothly on them, see Definition \ref{definition 2.1.} in next section. In \cite{Montalto} Montalto consider the stability of the inverse problem of determining a simple Riemannian metric together with the lower order coefficients of the second order hyperbolic initial-boundary value problem, from the information that is encoded in the global hyperbolic Dirichlet-to-Neumann map and prove H\"older type stability estimates near generic simple Riemannian metrics for the inverse problem of recovering simultaneously the metric and the electro-magnetic potentials. In \cite{BF}, the author and Dos Santos Ferriera proved stability estimates for the wave equation in determining a conformal factor close to $1$ and time independent potentials in simple geometries. We refer to this paper for a longer bibliography in the case of the wave equation. In \cite{LiuOk} Liu and Oksanen consider the problem to reconstruct a wave speed $c$ from acoustic boundary measurements modelled by the hyperbolic Dirichlet to Neumann map. They introduced a reconstruction formula for $c$ that is based on the Boundary Control method and incorporates features also from the complex geometric optics solutions approach. 
\smallskip

The cited results in above on this problem require that the corresponding Dirichlet to Neumann map be at least measured on a sufficiently large part  of the boundary of the spatial domain under consideration, if not on the whole boundary itself. However the stability by a local Dirichlet-to-Neumann map is not discussed comprehensively. In \cite{LO}, Lassas and Oksanen consider the inverse problem to determine a smooth compact Riemannian manifold with boundary from a restriction of the Dirichlet-to-Neumann operator for the wave equation on the manifold and show that the restriction of the Dirichlet to Neumann map determines the manifold $(\M,\g)$ uniquely, assuming that the wave equation is exactly controllable. The restriction here corresponds to the case where the Dirichlet data is supported on $\R\times\Gamma_1$ and the Neumann data is measured on $\R\times\Gamma_2$ for two parts $\Gamma_1,\Gamma_2$ of the boundary. 
\smallskip

An interesting related inverse problem is to determine the Riemannian metric $\g$ by the knowledge of $d_\g(x, y)$ the length of the geodesic joining $x$ and $y$ for all $(x, y)\in\Gamma\times\Gamma$. The obstruction to uniqueness for this inverse problem is similar to our inverse problem here, i.e. if $d_{\g_1}= d_{g_2}$,  then there exists a diffeomorphism $\psi:\Omega\to\Omega$  with $\psi=\textrm{Id}$ on $\Gamma$  such that $\psi^*\g_1=\g_2$. This problem is sometimes called the hodograph problem. This problem arose in geophysics in an attempt to determine the inner structure of the Earth by measuring the travel times of seismic waves. It goes back to Herglotz \cite{Heg} and Wiechert and Zoeppritz \cite{WZ}. Although the emphasis has been in the case that the medium is isotropic, the anisotropic case has been of interest in geophysics since it has been found that the inner core of the Earth exhibits anisotropic behavior \cite{CR}. In differential geometry this inverse problem has been studied because of rigidity questions and is known as the boundary rigidity problem. Stefanov and Uhlmann \cite{StU2009} solved this problem when Riemannian metrics are close to the Euclidean metric. They also gave a stability estimate for this inverse problem in \cite{StU2004}. 
\smallskip

For the Dirichlet to Neumann  map for an elliptic equation, the paper
by Calder\'on \cite{[Calderon]} is a pioneering work. We also refer to
Bukhgeim and Uhlamnn \cite{[Bukhgeim-Uhlmann]}, Hech-Wang
\cite{[Hech-Wang]}, Salo \cite{[Salo]} and Uhlmann \cite{[Uhlmann]}
as a survey. In \cite{[DKSU]} Dos Santos
Ferreira, Kenig, Sjostrand, Uhlmann prove that the knowledge of the
Cauchy data for the Schr\"odinger equation in the presence of
magnetic potential, measured on possibly very small subset of the
boundary, determines uniquely the magnetic field. In \cite{[Tzou]},
Tzou proves a log log-type estimate which show that the
magnetic field and the electric potential of the magnetic
Schr\"odinger equation depends stably on the Dirichlet to Neumann
map even when the boundary measurement is taken only on a subset
that is slightly larger than the half of the boundary. 
\smallskip

As for the stability of the wave equation in the Euclidian case, we also refer to \cite{[Sun]} and \cite{[IS]}; in those papers, the Dirichlet-to-Neumann map was considered on the  whole boundary. Isakov and Sun \cite{[IS]} proved that the difference in some subdomain of two coefficients is estimated by an operator norm of the difference of the corresponding local Dirichlet-to-Neumann maps, and that
the estimate is of H\"older type. Bellassoued, Jellali and Yamamoto \cite{BJY} considered the inverse problem of recovering a
time independent potential in the hyperbolic equation from the partial Dirichlet-to-Neumann map. They proved a logarithm stability estimate. Moreover in \cite{[Rakesh1]}  it is proved that if an unknown coefficient belongs to a given finite dimensional vector space, then the uniqueness
follows by a finite number of measurements on the whole boundary. In \cite{Bell-Bejoud}, Bellassoued and Benjoud used complex geometrical
optics solutions concentring near lines in any direction to prove that the Dirichlet-to-Neumann map determines uniquely the magnetic field induced by
a magnetic potential in a magnetic wave equation.

\subsection{Admissible manifolds}
We start by defining the Riemannian metrics in which we will give the stability results of this paper.
Let $(\M,\g)$ be a Riemannian manifold with boundary $\Gamma$. Fix an open non-empty subset $\mathcal{O}\subset\bar{\M}$ such that $\mathcal{O}$ is an arbitrary neighborhood of $\Gamma$ in $\bar{\M}$. We assume that $\bar{\M}$ is strictly convex in the usual sense and $\g$ satisfying 
\begin{equation}\label{1.10}
\g=e \quad \textrm{in}\,\mathcal{O},\quad \norm{\g-e}_{\mathcal{C}^k(\M)}\leq \varepsilon
\end{equation}
with some $k\geq 2$ and $\varepsilon>0$ sufficiently small and $e$ denotes the standard Euclidean metric. In other words, the domain has the homogeneous crust and the inhomogeneous core. Note that if $\g$ satisfies (\ref{1.10}) then $\bar{\M}$ is clearly geodesically convex with respect to $\g$, i.e. for any two distinct points $x,y\in\bar{\M}$ there is a unique geodesic connecting $x$ and $y$ which lies entirely in $\M$. We denote by
\begin{equation}\label{1.11}
S\M=\set{(x,\xi)\in T\M;\,\abs{\xi}_{\g}=1}, 
\end{equation}
 the sphere bundle  of $\M$.  Let $(x,\xi)\in S\M$, there exist a unique geodesic $\gamma_{x,\xi}$ associated to $(x,\xi)$ which is maxmimally defined on a finite interval $[\tau_-(x,\xi),\tau_+(x,\xi)]$, with $\gamma_{x,\xi}(\tau_\pm(x,\xi))\in\Gamma$. We define the geodesic flow $\Phi_t$ as following
\begin{equation}\label{1.12}
\Phi_t:S\M\to S\M,\quad \Phi_t(x,\xi)=(\gamma_{x,\xi}(t),\dot{\gamma}_{x,\xi}(t)),\quad t\in [\tau_-(x,\xi),\tau_+(x,\xi)],
\end{equation}
and $\Phi_t$ is a flow, that is, $\Phi_t\circ\Phi_s=\Phi_{t+s}$. We introduce the submanifolds of inner and outer vectors of $S\M$ as
\begin{equation}\label{1.13}
\p_{\pm}S\M =\set{(x,\xi)\in S\M,\, x \in \Gamma,\, \pm\seq{\xi,\nu(x)}> 0}.
\end{equation}
For $(x,\xi)\in\p_+ S\M$,  we denote by $\gamma_{x,\xi} : [0,\tau_+(x,\xi)] \to \M$ the maximal
geodesic satisfying the initial conditions $\gamma_{x,\xi}(0) = x$ and $\dot{\gamma}_{x,\xi}(0) = \xi$, where $\tau_+(x,\xi)$ be the  length of the geodesic. 
\smallskip

Now we can define the length and the  scattering relation
\begin{equation}\label{1.14}
\begin{array}{ccc}
\tau^+_\g : \p_-S\M &\longrightarrow  & \R_+  \\
(x,\xi) &\longmapsto & \tau_\g^+(x,\xi),
\end{array}
\,\,;\qquad 
\begin{array}{ccc}
\mathcal{S}_\g : \p_-S\M &\longrightarrow  & \p_+S\M \\
(x,\xi) &\longmapsto & \Phi_{\tau_\g^+}(x,\xi),
\end{array}
\end{equation}  
where $\tau_\g^+(x,\xi)> 0$ is the first moment, at which the unit speed geodesic through $(x,\xi)$ hits $\Gamma$ again. In other words, the map $\mathcal{S}_\g$ is obtained by measuring on the boundary the initial and final points and velocities of all geodesics passing through $\M$ (of course, $\g$ is assumed to be known on $\Gamma$). 
\smallskip

For $\g\in\mathcal{C}^k(\M)$, $k\geq 2$, be a simple metric in $\M$, using the global semi-geodesic coordinates in $\M$ already used in \cite{StU2}, \cite{StU3}, we can prove that there exists a $\mathcal{C}^{k-1}$-diffeomorphism $\psi:\M\to\psi(\M)$, such that  the pullback  $\psi^*\g$ of the metric $\g$ has the property
\begin{equation}\label{1.15}
(\psi^*\g)_{in}=\delta_{in},\quad i=1,\ldots,n.
\end{equation}
Let $\psi_1$ and $\psi_2$ be maps $y \to x$ related to the two simple metrics $\g_1$ and $\g_2$. The new metrics $\psi_1^*\g_1 := \tilde{\g}_1$ and $\psi_2^*\g_2 := \tilde{\g}_2$ have the form (\ref{1.15}). If we assume that $\Lambda_{\g_1}=\Lambda_{\g_2}$ we obtain that $\psi_1^{-1}(\p\M)=\psi_2^{-1}(\p\M)$. In other words, let $\tilde{\M}_1$ and $\tilde{\M}_2$ be the transformed domains of under the maps $\psi_1^{-1}$ and $\psi_2^{-1}$, then $\tilde{\M}_1=\tilde{\M}_2$  whenever $\Lambda_{\g_1}= \Lambda_{\g_2}$. Therefore, a diffeomorphism $\psi$  fixing the boundary is given by $\psi=\psi_1\circ\psi_2^{-1}$. Now if we aim to transform metrics into the form (\ref{1.15}) in proving the stability estimate, then $y_n\to x_n$ will be geodesics and if the domain is preserved, then the lengths of these geodesics must be the same, which equivalent to the Dirichlet to Neumann maps $\Lambda_{\g_1}$ and  $\Lambda_{\g_2}$ must be the same. However, this is not true for any two metrics. Therefore, if $\Lambda_{\g_1}\neq\Lambda_{\g_2}$, then both transforming metrics into the form (\ref{1.15}) and having the same transformed domain are impossible without additional information on metrics.
\smallskip

In order to formulate our result, we need to introduce some notations. Henceforth, let us set
\begin{equation}\label{1.16}
\Gamma_{\!-}=\set{x\in\Gamma,\,\,\seq{\nu(x),e_n}<0},
\end{equation}
where $e_n$ will denote always the standard unit vector of $\R^n$, with $e_n=(0,\dots,0,1)$. We introduce the admissible sets of metrics $\g$,
 \begin{definition}[Admissible metrics]
 Given $k\geq 2$ and $m_0>0$, and $\varepsilon>0$, we say that a couple of metrics tensors $(\g_1,\g_2)$ is $\mathcal{C}^k$-admissible if: For $\ell=1,2$, $\g_\ell\in \mathcal{C}^k(\M)$, $\norm{\g_{\,\ell}}_{\mathcal{C}^k(\M)}\leq M_0$, $\g_\ell$ satisfies (\ref{1.10}), and 
 \begin{equation}\label{admis}
\tau_{\g_1}^+(x,e_n)=\tau_{\g_2}^+(x,e_n),\quad \mathcal{S}_{\g_1}(x,e_n)=\mathcal{S}_{\g_2}(x,e_n),\quad \textrm{for all}\,\,x\in \Gamma_-.
 \end{equation}
 \end{definition}
In other words, the geodesics starting on $(x,e_n)\in\p_{\!-}S\M$, $x\in\Gamma_-$  with directions $e_n$ are a prior determined. 
\begin{itemize}
\item[(i-)]  Examples of manifolds satisfy (\ref{admis}) include a submanifold with boundary of $(\M_0,\mathrm{h}_\ell)\times(\R,dx_n\otimes dx_n)$, $\ell=1,2$ where $(\M_0,\mathrm{h}_\ell)$ are a compact  $(n-1)$-dimensional manifold.
\item[(ii-)] Any bounded strictly convex domain $\M$ in $\R^n$, endowed with a metrics $\g_\ell$, $\ell=1,2$, which in some coordinates has the form
$$
\g_\ell (x',x_n)=\begin{pmatrix}
\widetilde{\g}_\ell (x',x_n) & 0\\
0 &  1
\end{pmatrix},\quad\ell=1,2,
$$
are satisfy (\ref{admis}).
\end{itemize}
Let $\Gamma^\natural\subset \Gamma_{\!-}$ be an arbitrary non-empty relatively open subset of $\Gamma_{\!-}$. We set $\Sigma^\natural=(0,T)\times\Gamma^\natural$ and we consider the partial Dirichlet to Neumann map defined by (\ref{1.9}). We establish a stability result for the inverse problem consisting in the determination of the metric $\g$ from the partial Dirichlet to Neumann map $\Lambda_{\g}^\natural$.
\subsection{Main result}
In this subsection,  we will give the statement of our main result. In what follows, we will use $C=C(T,\M,M_0,k)$ to denote generic positive constants which may vary from line to line (unless otherwise stated). For $\alpha,\beta\in (0,1)$, we define the continuous function on $[0,\infty[$, 
\begin{equation}\label{1.17}
\Phi_{\alpha,\beta}(s)=\log\para{2+s^{-\alpha}}^{-\beta},\quad \Phi_{\alpha,\beta}(0)=0.
\end{equation}
The main result of this paper is the following theorem.
\begin{theorem}\label{Th1}
There exist $k \geq 2$, $\varepsilon > 0$, $\alpha,\beta\in(0,1)$, $T$ sufficiently large and $C > 0$ such that for any $\mathcal{C}^k$-admissible couple of metrics $(\g_1,\g_2)$,  there exists a diffeomorphism $\psi:\bar{\M}\to\bar{\M}$ with $\psi=\textrm{id}$ on $\Gamma$ such that
\begin{equation}\label{1.18}
\norm{\g_1-\psi^*\g_2}_{L^2(\M)}\leq C\Phi_{\alpha,\beta}\para{\norm{\Lambda^{\natural}_{\g_1}-\Lambda^{\natural}_{\g_2}}},
\end{equation}
where $C$ depends on $\M$, $T$, $M_0$, $k$  and $\varepsilon$.
\end{theorem}
By Theorem \ref{Th1}, we can readily derive the following uniqueness result
\begin{corollary}
There exist $k \geq 2$, $\varepsilon > 0$, $\mu\in(0,1)$, $T$ sufficiently large such that for any $\mathcal{C}^k$-admissible couple of metrics $(\g_1,\g_2)$ satisfying $\Lambda^{\natural}_{\g_1}=\Lambda^{\natural}_{\g_2}$, there exists a diffeomorphism $\psi:\bar{\M}\to\bar{\M}$ with $\psi=\textrm{id}$ on $\Gamma$ verifying $\g_1=\psi^*\g_2$.
\end{corollary}
This paper is organized as follows. Section 2 is devoted to the study of the geodesic ray transform of 2-symmetric tensor and give a change of coordinates in which the metrics takes a special form. In section 3  we build geometric optics solutions designed in accordance with our problem. Section 4 is devoted to the proof of the Theorem \ref{Th1}.

\section{Geodesic ray transform of 2-symmetric tensor and change of coordinates}
\setcounter{equation}{0}
In this section we first collect some  formulas needed in the rest of this paper and introduce the geodesical ray transform for $2$-symmetric tensors. 
\subsection{Preliminaries}
For this paper, we use many of the notational conventions in \cite{BF}. Let $(\M,\g)$ be a (smooth) compact Riemannian
manifold with boundary of dimension $n \geq 2$.
We refer to \cite{[Jost]} for the differential calculus of tensor fields on Riemannian manifolds. If we fix local coordinates $x=(x^1,\ldots,x^n)$ and let
$(\p_1,\dots,\p_n)$ denote the corresponding tangent vector fields, the inner product and the norm on the tangent
space $T_x\M$ are given by
$$
\g(X,Y)=\seq{X,Y}_\g=\sum_{j,k=1}^n\g_{jk}X^jY^k,\quad \abs{X}_\g=\seq{X,X}_\g^{1/2},\quad  X=\sum_{i=1}^nX^i\p_i,\,\, Y=\sum_{i=1}^n
Y^i\p_i.
$$
If $f$ is a $\mathcal{C}^1$ function on $\M$, we define the gradient of $f$ as the vector field $\nabla f$ such that, in local coordinates, we have
\begin{equation}\label{2.1}
\nabla_\g f=\sum_{i,j=1}^n\g^{ij}\frac{\p f}{\p x^i}\p_j.
\end{equation}
For  a smooth manifold $\M^*$ let $\psi:\M^*\to\M$ be a smooth map from $\M^*$ into $\M$. The pull back of the metric $\g$ is defined by
$$
(\psi^*\g)(X,Y)=\g(d\psi(X),d\psi(Y)),
$$
for any vector field $X,Y$ on $\M$.\\
The metric tensor $\g$ induces the Riemannian volume $\dv_\g=\sqrt{\abs{\g}}\dd x^1\wedge\cdots \wedge \dd x^n$. We denote by $L^2(\M)$ the completion
of $\mathcal{C}^\infty(\M)$ with respect to the usual inner product
$$
\seq{u,v}=\int_\M u(x)\overline{v(x)} \dv_\g,\qquad  u,v\in\mathcal{C}^\infty(\M).
$$
The Sobolev space $H^1(\M)$ is the completion of $\mathcal{C}^\infty(\M)$ with respect to the norm $\norm{\,\cdot\,}_{H^1(\M)}$,
$$
\norm{u}^2_{H^1(\M)}=\norm{u}^2_{L^2(\M)}+\norm{\nabla u}^2_{L^2(\M)}.
$$
Moreover, using covariant derivatives, it is possible to define coordinate invariant norms in $H^k(\M)$, $k\geq 0$.\\
Denote by $\dive X$ the divergence of a vector field $X\in H^1(\M,T\M)$ on $\M$, i.e. in local coordinates,
\begin{equation}\label{2.2}
\dive X=\frac{1}{\sqrt{\abs{\g}}}\sum_{i=1}^n\p_i\para{\sqrt{\abs{\g}}\,X^i},\quad X=\sum_{i=1}^nX^i\p_i.
\end{equation}
If $X\in H^1(\M,T\M)$ and $f\in H^1(\M)$, the Green's formula reads
\begin{equation}\label{2.3}
\int_\M\dive X\,f\dv=-\int_\M\seq{X,\nabla f} \dv+\int_{\Gamma}\seq{X,\nu} f\ds.
\end{equation}
where $\ds$ is the volume form of $\Gamma$. Then if $f\in H^1(\M)$ and $w\in H^2(\M)$, the following identity holds
\begin{equation}\label{2.4}
\int_\M\Delta w f\dv=-\int_\M\seq{\nabla w,\nabla f} \dv+\int_{\Gamma}\p_\nu w f \ds.
\end{equation}
The Riemannian scalar product on $T_x\M$ induces the volume form on $S_x\M$, denoted by $\dd \omega_x(\theta)$ and given by
\begin{equation}\label{2.5}
\dd \omega_x(\theta)=\sqrt{\abs{\g}} \, \sum_{k=1}^n(-1)^k\theta^k \dd \theta^1\wedge\cdots\wedge \widehat{\dd \theta^k}\wedge\cdots\wedge \dd \theta^n.
\end{equation}
As usual, the notation $\, \widehat{\cdot} \,$ means that the corresponding factor has been dropped.
We introduce the volume form $\dvv$ on the manifold $S\M$ by
$$
\dvv (x,\theta)=\dd\omega_x(\theta)\wedge \dv.
$$
By Liouville's theorem, the form $\dvv$ is preserved by the geodesic flow. 
Now, we recall the submanifolds of inner and outer vectors of $S\M$ given by
\begin{equation}\label{2.6}
\p_{\pm}S\M =\set{(x,\theta)\in S\M,\, x \in \Gamma,\, \pm\seq{\theta,\nu(x)}> 0}.
\end{equation}
Note that $\p_+ S\M$ and $\p_-S\M$ are compact manifolds with the same boundary $S\Gamma$, and $\p S\M = \p_+ S\M\cup \p_- S\M$. We denote by  $\mathcal{C}^\infty(\p_+ S\M)$ be the space of smooth functions on the manifold $\p_+S\M$. The corresponding volume form on the boundary $\p S\M =\set{(x,\theta)\in S\M,\, x\in\Gamma}$ is given by
$$
\dss=\dd\omega_x(\theta) \wedge \ds.
$$
Let $L^2_\mu(\p_+S\M)$ be the space of square integrable functions with respect to the measure $\mu(x,\theta)\dss$ with
$\mu(x,\theta)=\abs{\seq{\theta,\nu(x)}}$. This Hilbert space is endowed with the scalar product
\begin{equation}\label{2.7}
\para{u,v}_\mu=\int_{\p_+S\M}u(x,\theta) \overline{v}(x,\theta) \mu(x,\theta)\dss.
\end{equation}
\subsection{Geodesic ray transform of 2-symmetric tensor}
We denote by $\nabla$ the Levi-Civita connection on $(\M,\g)$. For a point $x \in \Gamma$, we define the second quadratic form of the boundary on the space $T_x\Gamma$ by
$$
\Pi(\xi,\xi)=\seq{\nabla_\xi\nu,\xi},\quad \xi\in T_x\Gamma.
$$
We say that the boundary is strictly convex if the form is positive-definite for all $x \in \Gamma$.
\smallskip

For $x\in \M$ and $\xi\in T_x\M$ we denote by $\gamma_{x,\xi}$ the unique geodesic starting at the point $x$ in the direction $\xi$.  The exponential map $\exp_x:T_x\M\To \M$ is given by
\begin{equation}\label{2.8}
\exp_x(v)=\gamma_{x,\xi}(\abs{v}),\quad \xi=\frac{v\,\,}{\abs{v}}.
\end{equation}
The Riemannian manifold $(\M,\, \g)$ is called a non-trapping manifold, if for each $(x,\theta) \in S\M$, the maximal geodesic $\gamma_{x,\theta}(t)$ satisfying the initial conditions $\gamma_{x,\theta}(0) = x$ and $\dot{\gamma}_{x,\theta}(0) = \theta$ is defined on a finite segment $[\tau_{-}(x,\theta), \tau_{+}(x,\theta)]$.  An important subclass of convex non-trapping manifolds are simple manifolds.
\begin{definition}(Simple manifold)\label{definition 2.1.}
We say that the Riemannian manifold $(\M,\g)$ (or more shortly that the metric $\g$) is simple, if $\Gamma= \p\M$ is
strictly convex with respect to $\g$, and for any $x\in \M$, the exponential map
$\exp_x:\exp_x^{-1}(\M)\To \M$ is a diffeomorphism.
\end{definition}
Any metric $\g$ on $\M$ so that $(\M, \g)$ is simple will be called a simple metric on $\M$. The second condition above hides the requirement that any two points $x, y$ in $\M$ are connected by a unique geodesic in $\M$ that depends smoothly on $x$, $y$. In particular, there are no conjugate points on any geodesic in $\M$. Note that if $(\M,\g)$ is simple, one can extend $(\M,\g)$ into another simple manifold $\M_{1}$ such that $\M \Subset \M_{1}$. 
\smallskip

Let $T^m_x\M$ be the space of tensors fields of type $m$ on $T_x\M$. We denote by $T^m\M$ the tensor bundle of type $m$. In the local coordinate system a $m$-tensor field $\mathfrak{t}$ can be written as
$$
\mathfrak{t}=t_{j_1,\dots,j_m}dx^{j_1}\otimes\dots\otimes dx^{j_m}.
$$
For each $x\in\M$, $T^m_x\M$ is endowed with an inner product as follows
$$
\seq{\mathfrak{t}_1 ,\mathfrak{t}_2}=\sum_{j_1,\dots,j_k=1}^n \mathfrak{t}_1(\p_{j_1},\dots,\p_{j_m})\mathfrak{t}_2(\p_{j_1},\dots,\p_{j_m}).
$$
Let $\mathcal{C}^\infty(\M,T^m\M)$ the space of the smooth $m$-tensor fields on $\M$. We denote by $L^2(\M,T^m\M)$ the space of square integrable $m$-tensors fields on $\M$ as the completion of $\mathcal{C}^\infty(\M,T^m\M)$ endowed with the following inner product
$$
\para{\mathfrak{t}_1,\mathfrak{t}_2}=\int_\M\seq{\mathfrak{t}_1,\overline{\mathfrak{t}_2}}\dv,\quad \mathfrak{t}_1,\mathfrak{t}_2 \in T^m\M.
$$
We denote by $\mathcal{C}_{\textrm{sym}}^\infty(\M,T^2\M)$ (resp. $L^2_{\textrm{sym}}(\M,T^2\M)$) the space of the smooth  (resp. the space of square integrable) symmetric $2$-tensor fields on $\M$. For $\mathfrak{s}\in \mathcal{C}_{\textrm{sym}}^\infty(\M,T^2\M)$, we have 
\begin{equation}\label{2.9}
\mathfrak{s}=\sum_{j,k=1}^n \mathrm{s}_{jk}dx^{j}\otimes dx^{k},\quad \mathrm{s}_{jk}=\mathrm{s}_{kj},\quad j,k=1,\ldots,n.
\end{equation}
It is well known that for a $\g$ smooth enough metric, each symmetric tensor $\mathfrak{s}\in L^2_{\textrm{sym}}(\M,T^2\M)$ admits unique orthogonal decomposition
\begin{equation}\label{2.10}
\mathfrak{s}=\mathfrak{s}^{\textrm{sol}}+\nabla_{\textrm{sym}}\mathrm{v},
\end{equation}
into a solnoidal tensor $\mathfrak{s}^{\textrm{sol}}\in L^2_{\textrm{sym}}(\M,T^2\M)$, i.e., $\delta\mathfrak{s}^{\textrm{sol}}=0$, and a potential tensor $\nabla_{\textrm{sym}}\mathrm{v}$, where $\mathrm{v}\in H_0^1(\M)$, with
\begin{equation}\label{2.11}
\nabla_{\textrm{sym}}\mathrm{v}=\frac{1}{2}\para{\nabla_j\mathrm{v}_k+\nabla_k\mathrm{v}_j },\quad 1\leq j,k,\leq n,
\end{equation}
and $\nabla_j$ is the covariant derivative in metric $\g$.\\
The ray transform of symmetric $2$-tensor on a simple Riemannian manifold $(\M,\g)$ is the linear operator:
$$
\I_\g:\mathcal{C}_{\textrm{sym}}^\infty(\M,T^2\M) \To \mathcal{C}^\infty(\p_+S\M)
$$
defined by
$$
\I_\g (\mathfrak{s})(x,\theta)=\int_{\gamma_{x,\theta}}f=\sum_{j,k=1}^n\int_0^{\tau_+(x,\theta)} \mathrm{s}_{jk}(\gamma_{x,\theta}(t))\dot{\gamma}^j_{x,\theta}(t)\dot{\gamma}^k_{x,\theta}(t)dt,
$$
where $\gamma_{x,\theta}: [0,\tau_+(x,\theta)]\to\M$ is a maximal geodesic satisfying the initial conditions $\gamma_{x,\theta}(0)=x$ and $\dot{\gamma}_{x,\theta}(0)=\theta$.
It is easy to see that $\I_\g(\nabla_{\textrm{sym}}\mathrm{v})=0$ for any smooth vector field $\mathrm{v}$ in $\M$ with $\mathrm{v}=0$ on $\Gamma$.
\begin{definition}($s$-injective) Let $(\M,\g)$ be a simple Riemannian manifold. We say that $\I_\g$ is $s$-injective in $\M$, if $\I_\g \mathfrak{s}=0$ and $\mathfrak{s}\in L^2(\M)$ imply $\mathfrak{s}=\nabla_{\mathrm{sym}}\mathrm{v}$ with some vector field $\mathrm{v}\in H^1(\M)$ with $\mathrm{v}=0$ on $\Gamma$.
\end{definition}
In other words, if $\mathfrak{s}\in H^1_{\textrm{sym}}(\M,T^2\M)$ and $\I_\g(\mathfrak{s})=0$ implies $\mathfrak{s}^{\textrm{sol}}=0$, i.e., $\mathfrak{s}=\nabla_{\textrm{sym}}\mathrm{v}$ with some vector field $\mathrm{v}$ vanishing on $\Gamma$. So we have
\begin{equation}\label{2.12}
\abs{\I_\g(\mathfrak{s})(x,\theta)}=\abs{\I_\g(\mathfrak{s}^{\textrm{sol}})(x,\theta)}\leq C\norm{\mathfrak{s}^{\textrm{sol}}}_{\mathcal{C}^0}, \quad \mathfrak{s}\in \mathcal{C}_{\textrm{sym}}^0(\M,T^2\M).
\end{equation}
The ray transform $\I_\g$ of symmetric $2$-tensor on a simple Riemannian manifold  is a bounded operator from $L^2_{\textrm{sym}}(\M,T^2\M)$ into $L^2_\mu(\p_+S\M)$ and  can be extend to the bounded operator
\begin{equation}\label{2.13}
\I_\g:H_{\textrm{sym}}^k(\M,T^2\M)\To H^k (\p_+S\M).
\end{equation}
Now, we recall some properties of the  ray transform of symmetric $2$-tensors on a simple Riemannian manifold proved in \cite{StU3}. Let $(\M,\g)$ be a simple metric, we assume that $\g$ extends smoothly as a simple metric on $\M_1\Supset \M$ and let $\I_\g^*:L^2_\mu(\p_+S\M)\to L^2_{\textrm{sym}}(\M,T^2\M)$ the adjoint of $\I_\g$. We denote $N_\g=\I_\g^*\I_\g$. It is well known, see Stefanov Uhlamnn \cite{StU3},  that there exists $k_0$ such that for each $k\geq k_0$, the set $\mathcal{G}^k(\M)$ of simple $\mathcal{C}^k(\M)$ metrics in $\M$ for which $\I_\g$ is s-injective is open and dense in the $\mathcal{C}^k(\M)$ topology. Moreover, for any $\g\in \mathcal{G}^k(\M)$,
\begin{equation}\label{2.14}
\norm{\mathfrak{s}^{\textrm{sol}}}_{L^2(\M)}\leq C\norm{N_\g(\mathfrak{s})}_{H^2(\M_1)},
\end{equation}
for any $\mathfrak{s}\in L^2_{\textrm{sym}}(\M,T^2\M)$, with a constant $C>0$ that can be chosen locally uniform in $\mathcal{G}^k$ in the $\mathcal{C}^k(\M)$ topology.
\smallskip

If $\M'$ is an open set of the simple Riemannian manifold $(\M_{1},\g)$, the normal operator $N_\g$ is an elliptic  pseudo-differential operator of order $-1$ on $\M'$. Therefore for each $k\geq 0$ there exists a constant $C_k>0$ such that for all $\mathfrak{s}\in H^k_{\textrm{sym}}(\M,T^2\M)$ compactly supported in $\M'$
\begin{equation}\label{2.15}
\norm{N_\g(\mathfrak{s})}_{H^{k+1}(\M_{1})}\leq C_k\norm{\mathfrak{s}^{\textrm{sol}}}_{H^k(\M')}.
\end{equation}
Since for $\g=e$, $\I_e$ is $s$-injective, then for $\varepsilon>0$ sufficiently small and any metric tensor $\g$ satisfy (\ref{1.10}), we have $(\M,\g)$ is simple and $\I_\g$ is $s$-injective.
\subsection{Change of coordinates}
In this section we will construct a diffeomorphism $\psi$ that fixes the boundary $\Gamma$. Let $(\M,\g)$ be an $n$-dimensional, $n\geq 2$, compact Riemannian manifold with smooth boundary $\Gamma$, with $\M\subset\R^n$ is convex, and bounded domain. We first assume throughout this section that $\g$ satisfies (\ref{1.10}). In view of (\ref{1.10}), we can extend $\g$ outside of $\M$, still denoted by $\g$, by setting $\g=e$ on 
$\R^n\backslash\M$ so that $\g\in\mathcal{C}^k(\R^n)$ and
\begin{equation}\label{2.16}
\norm{\g-e}_{\mathcal{C}^{k}(\R^n)}\leq \varepsilon.
\end{equation}
The Hamiltonian related to $\g$, is
\begin{equation}\label{2.17}
H(x,\xi)=\frac{1}{2}\para{\abs{\xi}^2_{\g}-1}=\frac{1}{2}\big(\sum_{i,j=1}^n\g^{ij}\xi_i\xi_j-1\big).
\end{equation}
Consider the following Hamiltonian system:
\begin{equation}\label{2.18}
\left\{\begin{array}{lll}
\displaystyle \dot{x}_k(s)=\sum_{j=1}^n\g^{kj}\xi_j,\quad \dot{\xi}_k(s)=-\frac{1}{2}\sum_{i,j=1}^n\frac{\p\g^{ij}}{\p x_k}\xi_i\xi_j,\quad k=1,\ldots,n,\cr
x(-\varrho)=(z,-\varrho)\quad \xi(-\varrho)=e_n.
\end{array}
\right.
\end{equation} 
Here $z\in\R^{n-1}$, $\varrho>0$ is such that $\g=e$ for $\abs{x}>\varrho$ and the solution $\para{x(s),\xi(s)}=\para{x_\g(z,s),\xi_\g(z,s)}$ depends on the parameter $z$. We remark that if $\g=e$, then $x_e(s)=(z,s-\varrho)$ and $\xi_e(s)=e_n$. As in \cite{StU3}, it follows from the estimate (\ref{2.16}) that, for any $a>0$ fixed, there exist $C>0$ such that
\begin{equation}\label{2.19}
\norm{x_\g-x_e}_{\mathcal{C}^{k-1}(\R^{n-1}\times [0,a])}+\norm{\xi_\g-e_n}_{\mathcal{C}^{k-1}(\R^{n-1}\times [0,a])}\leq C\varepsilon.
\end{equation}
In particular, (\ref{2.19}) implies that under the smallness assumption (\ref{2.16}) the Hamiltonian flow is non-trapping for small $\varepsilon$, more precisely, $x_\g(z,s)\notin B_\varrho:=\set{x\in\R^n, \abs{x}<\varrho}$, for $s>a$ with some $a>0$. Moreover, the mapping $(z,s)\to x_\g(z,s)$ is a $\mathcal{C}^{k-1}$-diffeomorphism on $\set{z\in\R^{n-1},\,\,\abs{z}\leq 2\varrho}\times [0, a]$ and its range covers $B_\varrho$ provided that $\varepsilon$ is small enough.
\smallskip

Introduce new coordinates $y=(z,s)$. Then the map $\psi:\M^*=\psi^{-1}(\M)\to \M$, $y\to x_\g(z,s)$ is close to $\textrm{id}$ in the $\mathcal{C}^{k-1}$-topology for small $\varepsilon$ and therefore is a diffeomorphism. This in particular implies that for $\varepsilon$ small enough
\begin{equation}\label{2.20}
\norm{\psi-\textrm{id}}_{\mathcal{C}^{k-1}}\leq C\varepsilon.
\end{equation}
In the new coordinates $\g$  satisfies the following lemma
\begin{equation}\label{2.21}
(\psi^*\g)_{in}=\delta_{in},\quad i=1,\ldots,n.
\end{equation}
Now let $(\g_1,\g_2)$ a couple of $\mathcal{C}^k$-admissible metrics on $\M$. Let us denote $\psi_\ell:\M^*_\ell\to \M$, $y\to x_{\g_\ell}(z,s)$ the diffeomorphism related to $\g_\ell$ as above, where $\M^*_\ell=\psi_\ell^{-1}(\M)$, $\ell=1,2$. We state the following 
\begin{lemma}\label{L2.3}
Let $(\g_1,\g_2)$ a couple of $\mathcal{C}^k$-admissible metrics on $\M$. Then we have
\begin{equation}\label{2.22}
\psi_1^{-1}(\Gamma)=\psi_2^{-1}(\Gamma).
\end{equation}
So $\psi_1^{-1}$ and $\psi_2^{-1}$ map $\M$ onto a new domain $\M^*=\M^*_1=\M^*_2$ with smooth boundary $\Gamma^*=\p\Omega^*$ and
\begin{equation}\label{2.23}
\Gamma_-^*:=\set{x\in \Gamma^*,\,\,\seq{\nu(x),e_n}<0}=\set{x\in \Gamma,\,\,\seq{\nu(x),e_n}<0}=\Gamma_-.
\end{equation}
Moreover, we have
\begin{equation}\label{2.24}
\psi_1^*\g_1=\psi_2^*\g_2\quad \textrm{in}\,\,\mathcal{O}^*\subset \M^*,\quad \norm{\psi_\ell^*\g_\ell-e}_{\mathcal{C}^{k-2}(\M^*)}\leq C\varepsilon,\quad \ell=1,2,
\end{equation}
for some positive constant $C$ and a neighborhood $\mathcal{O}^*$ of $\Gamma^*=\p\M^*$.
\end{lemma}
\begin{proof}
We denote
\begin{equation}\label{2.25}
\Gamma_{\pm}^{(\ell)}=\set{y\in\p\M^*_\ell,\,\, \pm\seq{\nu^*_\ell(y),e_n}>0},\quad \ell=1,2,
\end{equation}
where $\nu^*_\ell$ is the outer normal to $\p\M^*_\ell$ at $y$ relative to $\psi_\ell^*\g_\ell$.\\
First, it is clear that $\psi_1^{-1}(x)=\psi_2^{-1}(x)=x$, for any $x\in\Gamma_-=\set{x\in\Gamma,\,\seq{\nu(x),e_n}<0}$, since $\g_\ell=e$ before $s\to x_{\g_\ell}(s)$ enter $\Omega_0:=\Omega\backslash\mathcal{O}$. Then, we  obtain 
\begin{equation}\label{2.26}
\Gamma_{-}^{(1)}=\Gamma_{-}^{(2)}=\Gamma_{-}.
\end{equation}
Let now $y_1^+=(z,s_1^+)\in \Gamma^{(1)}_{+}$ with $z\in\R^{n-1}$. We denote by $\varsigma$ the geodesics in $\M^*_\ell$ connecting $y_1^-:=(z,s_1^-) $ and $y_1^+=(z,s_1^+)$ for some point $y_1^-=(z,s_1^-)\in \Gamma^{(1)}_{\!-}$. Note that $\varsigma$ is straight line in $\M^*_\ell$ with $s$ is being the arclength in both metrics $\psi_\ell^*\g_\ell$, $\ell=1,2$. Since we have also by (\ref{2.26}) that $y_1^-=(z,s_1^-)\in \Gamma^{(2)}_{\!-}$, then  $\varsigma$ relies $y_1^-=(z,s_1^-)$ to $y_2^+=(z,s_2^+)\in \Gamma^{(2)}_{\!+}$ for some point $y_2^+$. We denote
\begin{equation}\label{2.27}
x=\psi_1(y_1^-)=\psi_2(y_1^-)\in\Gamma_-,\quad\xi=\psi_{1*}e_n=\psi_{2*}e_n=e_n,
\end{equation}
where $\psi_{\ell*}: T_{y_1}\M^*_\ell \to T_{\psi_\ell(y_1)}\M$, $\ell=1,2$, is the differential of $\psi_\ell$ at $y_1$. Then we have 
\begin{equation}\label{2.28}
\mathcal{S}_{\g_\ell}(x,e_n)=(\gamma_{x,\xi}^{(\ell)}(\tau_+^{(\ell)}), \dot{\gamma}_{x,\xi}^{(\ell)}(\tau_+^{(\ell)})):=(x_{\ell}^+,\eta_{\ell}), \quad \tau_+^{(\ell)}=\tau^+_{\g_\ell}(x,e_n)=d_{\g_\ell}(x,x_{\ell}^+),
\end{equation}
where $\gamma_{x,\xi}^{(\ell)}=\psi_\ell(\varsigma)$ and for some $x_{\ell}^+\in\Gamma$.\\
Since $x\in \Gamma_-$. By the assumption (\ref{1.17}), we obtain
\begin{equation}\label{2.29}
\tau_+^{(1)}(x,e_n)=\tau_+^{(2)}(x,e_n),\quad  \mathcal{S}_{\g_1}(x,e_n)=\mathcal{S}_{\g_2}(x,e_n).
\end{equation}
In particular, $x_1^+=x_2^+$ and $d_{\g_1}(x,x_1^+)=d_{\g_2}(x,x_2^+)$. Note that the pull-back preserves the arclength. Therefore
\begin{equation}\label{2.30}
d_{\psi_1^*\g_1}(x,y_1^+)=d_{\g_1}(x,x_1^+)=d_{\g_2}(x,x_2^+)=d_{\psi_2^*\g_2}(x,y_2^+).
\end{equation}
Since $s$ is the arclength parameter in both metrics $\psi_1^*\g_1$ and $\psi_2^*\g_2$, we get $s_1^+=s^+_2$, that is $y_+^{(1)}=y_+^{(2)}\in
\Gamma^{(2)}_+$. Therefore
$$
\Gamma^{(1)}_+=\Gamma^{(2)}_+,\quad \textrm{and}\quad \psi_1^{-1}=\psi_2^{-1}\,\,\textrm{in}\,\,\Gamma.
$$
In view of this fact and the estimate (\ref{2.20}), we can derive the following estimate:
\begin{equation}\label{2.31}
\norm{\psi_\ell^*\g_\ell-e}_{\mathcal{C}^{k-2}(\M^*)}\leq C\varepsilon.
\end{equation}
Note that the $\mathcal{C}^{k-2}$-smoothness is due to the fact the pull-back $\psi_\ell^*\g_\ell$ contains the Jacobian of $\psi_\ell$, $\ell=1,2$.
Using the fact used above for the couple of $\mathcal{C}^k$-admissible metrics $(\g_1,\g_2)$, it is readily seen that
\begin{equation}\label{2.32}
\psi_1^{-1}=\psi_2^{-1},\quad \textrm{on}\,\,\mathcal{O}. 
\end{equation}
Therefore $\psi_1^*\g_1$ and $\psi_2^*\g_2$ will satisfy
\begin{equation}\label{2.33}
\psi_1^*\g_1=\psi^*_2\g_2,\quad\textrm{on}\,\,\mathcal{O}^*,
\end{equation}
form some neighborhood $\mathcal{O}^*$ of $\Gamma^*=\p\M^*$.\\
This completes the proof of the Lemma.
\end{proof}
Let $\psi_1$ and $\psi_2$ as above.  We have $\psi_1^{-1}(B_\varrho\setminus\M)=\psi_2^{-1}(B_\varrho\setminus\M)$. Then $\psi=\psi_2\circ\psi_1^{-1}$, is a diffeomorphism from $\bar{\M}$ into $\bar{\M}$, and mapping the unit speed geodesics for $\g_1$ normal to $\Gamma$ into unit speed geodesics for $\g_2$ normal to $\Gamma$ and $\Lambda^{\natural}_{\g_2}=\Lambda^{\natural}_{\psi^*\g_2}$. 
\begin{lemma}\label{L2.4}
Let $\g_\ell\in\mathcal{C}^k(\M)$ and $\psi_\ell$, $\ell=1,2$ as above. Let $\Gamma^\natural\subset\Gamma_{\!-}=\Gamma^*_{\!-}$. Then we have
\begin{equation}\label{2.34}
\norm{\Lambda^\natural_{\psi_1^*\g_1}-\Lambda^\natural_{\psi_2^*\g_2} }_{\mathcal{L}(H^{1,1}_0(\Sigma^*),L^2(\Sigma^\natural))}\leq C \norm{\Lambda^\natural_{\g_1}-\Lambda^\natural_{\g_2} }_{\mathcal{L}(H^{1,1}_0(\Sigma),L^2(\Sigma^\natural))},
\end{equation}
Here $\Sigma^*=(0,T)\times\Gamma^*$.
\end{lemma}
\begin{proof}
For $f\in H^{1,1}_0(\Sigma)$, let $u_\ell$, $\ell=1,2$, solve the following initial boundary value problem
\begin{equation}\label{2.35}
\left\{
\begin{array}{llll}
\para{\partial^2_t-\Delta_{\g_\ell}}u_\ell=0  & \textrm{in }\; Q:=(0,T)\times\M,\cr
u_\ell(\cdot,0 )=\p_tu_\ell(\cdot,0)=0 & \textrm{in }\; \M ,\cr
u_\ell=f & \textrm{on } \; \Sigma=(0,T)\times\Gamma.
\end{array}
\right.
\end{equation}
Then $v_\ell=\psi_\ell^*u_\ell=u_\ell\circ\psi_\ell$ solves
\begin{equation}\label{2.36}
\left\{
\begin{array}{llll}
\para{\partial^2_t-\Delta_{\psi_\ell^*\g_\ell}}v_\ell=0  & \textrm{in }\; Q^*:=(0,T)\times \M^*,\cr
v_\ell(0,\cdot)=\p_tv_\ell(0,\cdot)=0 & \textrm{in }\; \M^* ,\cr
v_\ell=\psi_\ell^*f:=f\circ\psi_\ell & \textrm{on } \; \Sigma^*=(0,T)\times\Gamma^*.
\end{array}
\right.
\end{equation}
So using, the fact that $\Gamma^\natural\subset\Gamma_-=\Gamma^*_-$, and $\psi_\ell=\textrm{id}$ near $\Gamma^\natural$, $\ell=1,2$, we obtain
\begin{equation}\label{2.37}
\Lambda^\natural_{\g_\ell}(f)=\p_\nu u_{\ell |\Sigma^\natural}=\p_\nu v_{\ell |\Sigma^\natural}=\Lambda^\natural_{\psi_\ell^*\g_\ell}(f^*),\quad f^*=f\circ\psi_1=f\circ\psi_2.
\end{equation}
Therefore
\begin{equation}\label{2.38}
\norm{(\Lambda^\natural_{\psi_1^*\g_1}-\Lambda^\natural_{\psi_2^*\g_2}) f^*}=\norm{(\Lambda^\natural_{\g_1}-\Lambda^\natural_{\g_2}) f}\leq \norm{\Lambda^\natural_{\g_1}-\Lambda^\natural_{\g_2}}\norm{f}_{H^{1,1}(\Sigma)}.
\end{equation}
We may assume that, without loos of generality, $\Omega=\set{x\in\R^n,\,\rho(x)<0}$, where $\rho:\R^n\to\R$ is $\mathcal{C}^1$ such that $d\rho\neq 0$ at $\Gamma=\set{x\in\R^n,\,\,\rho(x)=0}$. We have
\begin{equation}\label{2.39}
\int_{\Gamma}\abs{f(t,x)}^2d\sigma(x)=\lim_{\delta\to 0}\int_\M\theta\para{\frac{\rho(x)}{\delta}}\delta^{-1}\abs{u_1(t,x)}^2\abs{d\rho(x)}dx.
\end{equation}
That limit does depend on the choice of the defining function $\rho$ of $\Gamma=\p\M$, nor the choice of the function $\theta\in\mathcal{C}^\infty_0(\R,\R^+)$,  satisfying, $\int_\R\theta(s)ds=1$. Consider $\rho^*=\rho\circ\psi_1$ the defining function of $\Gamma^*=\p\M^*$:
$$
\Gamma^*=\set{x\in\R^n,\,\,\rho^*(x)=0},
$$
 and taking the variable change $y=\psi_1(x)$ the right integral in (\ref{2.39}), can be write as
\begin{equation}\label{2.40}
\int_\M\theta\para{\frac{\rho(x)}{\delta}}\delta^{-1}\abs{u_1(t,x)}^2\abs{d\rho(x)}dx=\int_{\M^*}\theta\para{\frac{\rho^*(x)}{\delta}}\delta^{-1}\abs{v_1(t,x)}^2\abs{d\rho^*(x)}dx.
\end{equation}
Taking $\delta\to 0$, we obtain from (\ref{2.40})
\begin{equation}\label{2.41}
\int_{\Gamma}\abs{f(t,x)}^2d\sigma(x)=\int_{\Gamma^*}\abs{f^*(t,x)}^2d\sigma^*(x).
\end{equation}
We repeat the same argument to the derivatives of $f$, we obtain, $\norm{f}_{H^{1,1}(\Sigma)}=\norm{f^*}_{H^{1,1}(\Sigma^*)}$. We deduce that from (\ref{2.38})
$$
\norm{(\Lambda^\natural_{\psi_1^*\g_1}-\Lambda^\natural_{\psi_2^*g_2}) f^*}\leq C\norm{\Lambda^\natural_{\g_1}-\Lambda^\natural_{\g_2}}\norm{f^*}_{H^{1,1}(\Sigma^*)}.
$$
This completes the proof. 
\end{proof}
Therefore, it suffices to consider the problem on $\M^*$ for metrics with the form (\ref{2.21}) and satisfying (\ref{2.24}). Our goal is to show
\begin{equation}\label{2.42}
\norm{\psi_1^*\g_1-\psi_2^*\g_2}_{L^2(\M^*)}\leq C\Phi_{\alpha,\beta}\para{\norm{\Lambda^{\natural}_{\psi_1^*\g_1}-\Lambda^{\natural}_{\psi_2^*\g_2}}},
\end{equation}
and recall that
\begin{equation}\label{2.43}
\norm{\g_1-\psi^*\g_2}_{L^2(\M)}\leq C\norm{\psi_1^*\g_1-\psi_2^*\g_2}_{L^2(\M^*)},
\end{equation}
where $\psi=\psi_2\circ\psi_1^{-1}$. Then the estimate (\ref{1.18}) follows easily from (\ref{2.42}), (\ref{2.43}) and (\ref{2.34}).\\
Hence,  we denote $\psi_\ell^*\g_\ell$ by $\g_\ell$, and $\M^*$ by $\M$. From now on we will follow this notation.
\section{Geometrical optics solutions and integral identity}
\setcounter{equation}{0}
As stated in the previous section, we aim to establish (\ref{2.42}). To this end, we first proceed to the construction of geometrical optics solutions to the  wave equation. With an abuse of notations, from now on we will replace the pulled back metric $\psi_1^*\g_1$ and $\psi_2^*\g_2$ by $\g_1$ and $\g_2$ respectively, and the domains $\M^*$ and $\mathcal{O}^*$ by $\M$ and $\mathcal{O}$ respectively. Now we summarize the properties of $\g_1$ and $\g_2$. Following (\ref{2.21}), and (\ref{2.24}), we have
\begin{equation}\label{3.1}
(\g_\ell)_{in}=\delta_{in},\quad  i=1,\ldots,n,\,\,\ell=1,2,\quad \g_\ell\in\mathcal{C}^{k-1}(\M).
\end{equation}
and
\begin{equation}\label{3.2}
\g_1=\g_2\,\,\textrm{in}\,\,\mathcal{O}\subset \M,\quad \norm{\g_\ell-e}_{\mathcal{C}^{k-2}(\M)}\leq C\varepsilon,\quad \ell=1,2.
\end{equation}
 In view of conditions (\ref{3.2}), we can extend $\g_1$ and $\g_2$ to be $\mathcal{C}^{k-2}$-smooth metrics on $\M_1 \Supset \M$ with $\g_1=\g_2$ on $\M_1\backslash\M$ such that $\M_1$ is geodesically convex with respect to any one of them. Even increase the constant $C$ in (\ref{3.2}), we can assume that the $\mathcal{C}^{k-2}(\M_{1})$-norms of $(\g_\ell-e)$, $\ell=1,2$, are also bounded by $C\varepsilon$.
\smallskip

In order to prove the main results, we first need some basic estimates of the solutions for the wave equation. We  consider the following initial boundary value problem for the wave equation
\begin{equation}\label{3.3}
\left\{
\begin{array}{llll}
\para{\partial_t^2-\Delta_\g}r(t,x)=z(t,x)  & \textrm{in }\,\,Q,\cr
r(0,x)=0,\quad\p_tr(0,x)=0 & \textrm{in }\,\,\M,\cr
r(t,x)=0 & \textrm{on } \,\, \Sigma.
\end{array}
\right.
\end{equation}
We know this problem is well-posed, since we have the following existence and uniqueness result, see \cite{BF}.
\begin{lemma}\label{L3.1}
Let $T>0$, $\g\in \mathcal{C}^2(\M)$. Assuming $z\in H^1(0,T;L^2(\M))$ such that $z(0,\cdot)=0$ in $\M$, then there exists a unique solution
$r$ to \eqref{3.3} such that
 $$
 r\in \mathcal{C}^2([0,T];L^2(\M))\cap \mathcal{C}^1([0,T];H^1_0(\M))\cap \mathcal{C}([0,T]; H^2(\M)).
 $$ 
 Furthermore, there is a constant $C>0$ such that
\begin{equation}\label{3.4}
\norm{\p_tr(t,\cdot)}_{L^2(\M)}+\norm{\nabla r(t,\cdot)}_{L^2(\M)}\leq C\norm{z}_{L^2(Q)},
\end{equation}
and
\begin{equation}\label{3.5}
\norm{\p_t^2 r(t,\cdot)}_{L^2(\M)}+\norm{\p_t r(t,\cdot)}_{H^1(\M)}+\norm{r(t,\cdot)}_{H^2(\M)}\leq C\norm{z}_{H^1(0,T;L^2(\M))}.
\end{equation}
\end{lemma}
As mentioned above, the proof of Theorem \ref{Th1} is based on the use of geometric optical solutions. We use the oscillating solutions of the form
$$
u(t,x)=\sum_{j=1}^k a_j(t,x;h)e^{i\varphi_j(t,x)/h}+r_h(t,x), \quad (t,x)\in Q
$$
with $h\in (0,h_0)$ a small parameter, $r_h$ is remainder term that admits a decay with respect to the parameter $h$, and $\varphi_j$ are real valued functions, $j=1,\ldots,k$. Inspired by Bellassoued and Ferriera \cite{BF}, Bellassoued \cite{B} and Bellassoued and Rezig \cite{BR}, we use these particular solutions to prove that the hyperbolic inverse boundary value problem reduces to the problem to invert, in some sens, the geodesic ray transform on $(\M,\g)$. As indicated above, the assumption that on the smallness on the metrics $(\M, \g_\ell)$, $\ell=1,2$, guarantees that this metrics is simple and the geodesic ray transform transform is indeed $s$-invertible.
\smallskip

We introduce now the space $\mathcal{A}(Q)=H^1(0,T;H^2(\M))\cap H^3(0,T;L^2(\M))$, equipped with the norm
\begin{equation}\label{3.6}
\nr(a)=\norm{a}_{H^1(0,T;H^2(\M))}+\norm{a}_{H^3(0,T;L^2(\M))},\quad a\in \mathcal{A}(Q).
\end{equation}
\subsection{Solutions for the backward wave equation}
We suppose, for a moment, that we are able to find a function $\ek_1\in\mathcal{C}^2(\M)$ which satisfies the eikonal equation
\begin{equation}\label{3.7}
\abs{\nabla_{\g_1}\ek_1}_{\g_1}^2=\sum_{i,j=1}^n\g_1^{ij} \p_i\ek_1\p_j \ek_1=1,\qquad \forall x\in \M,
\end{equation}
and assume that there exist a function $a_1\in \mathcal{A}(Q)$ which solves the transport equation
\begin{equation}\label{3.8}
\mathcal{L}_{\g_1,\ek_1}a_1(t,x):=\p_t a_1+\sum_{j,k=1}^n \g_1^{jk}\p_j\ek_1\p_k a_1+\frac{1}{2} (\Delta_{\g_1} \ek_1)a_1=0,\qquad \forall t\in\R,\, x\in\M,
\end{equation}
which satisfies for $T>T^*>\textrm{Diam}(\M)+2\delta$
\begin{equation}\label{3.9}
a_1(t,x)|_{t\leq \delta}=a_1(t,x)|_{t\geq T^*-\delta}=0,\quad \forall x\in \M.
\end{equation}
\begin{lemma}\label{L3.2}
 Let  $\g_1\in \mathcal{C}^k(\M)$, $k\geq 2$, with $\norm{\g}_{\mathcal{C}^k(\M)}\leq M_0$. Choose $a_1\in\mathcal{A}(Q)$ solve  (\ref{3.8})-(\ref{3.9}) and $\ek_1\in\mathcal{C}^2(\M)$ satisfy (\ref{3.7}). Then for all $h>0$ small enough, there exists a solution 
 $$
 u_1(t,x;h)\in \mathcal{C}^2([0,T^*];L^2(\M))\cap\mathcal{C}^1([0,T^*];H^1(\M))\cap\mathcal{C}([0,T^*];H^2(\M))
 $$ 
 of the wave equation
\begin{equation*}
(\p^2_t-\Delta_{\g_1})u_1=0,\quad \textrm{in}\quad Q^*:=(0,T^*)\times\M,
\end{equation*}
with the final condition
\begin{equation*}
u_1(T^*,x)=\p_tu_1(T^*,x)=0,\quad \textrm{in}\quad \M,
\end{equation*}
 of the form
\begin{equation}\label{3.10}
u_1(t,x)=a_1(t,x)e^{i(\ek_1(x)-t)/h}+r_{1,h}(t,x),
\end{equation}
the remainder $r_{1,h}(t,x)$ is such that
\begin{align*}
r_{1,h}(t,x)&=0,\quad (t,x)\in \Sigma , \\
r_{1,h}(T^*,x)=\p_t r_{1,h}(T^*,x)&=0,\quad x\in \M.
\end{align*}
Furthermore, there exist $C>0$, $h_0>0$ such that, for all $h \leq h_0$ the following estimates hold true.
\begin{equation}\label{3.11}
\sum_{k=0}^2 \sum_{j=0}^{k} h^{k-1} \|\p_t^{j} r_{1,h}(t,\cdot)\|_{H^{k-j}(\M)}\leq C\nr(a_1).
\end{equation}
The constant $C$ depends only on $T^*$, $\M$ and $M_0$ (that is $C$ does not depend on  $h$ and $\varepsilon<1$). 
\end{lemma}
\begin{proof}
Let $r_{1,h}(t,x)$ solves the following homogeneous boundary value problem
\begin{equation}\label{3.12}
\left\{
\begin{array}{llll}
\para{\partial^2_t-\Delta_{\g_1}}r(t,x)=z_{1,h} (t,x)
& \textrm{in }\,\, Q^*,\cr
r(T^*,x)=\p_tr(T^*,x)=0,& \textrm{in
}\,\,\M,\cr
r(t,x)=0 & \textrm{on} \,\, \Sigma,
\end{array}
\right.
\end{equation}
where the source term $z_{1,h}$ is given by 
\begin{equation}\label{3.13}
z_{1,h}(t,x)=-\para{\partial_t^2-\Delta_{\g_1}}\para{a_1(t,x)e^{i(\ek_1-t)/h}}.
\end{equation}
To prove our Lemma it would be enough to show that $r$ satisfies the estimates (\ref{3.11}). By a simple computation, we have
\begin{align}\label{3.14}
-z_{1,h} (t,x)&=e^{i(\ek_1(x)-t)/h}\para{\p_t^2-\Delta_{\g_1}}a_1(t,x)-2ih^{-1} e^{i(\ek_1(x)-t)/h}\mathcal{L}_{\g_1,\ek_1}a_1(t,x)\cr
&\quad-h^{-2} a_2(t,x) e^{i(\ek_1(x)-t)/h}\para{1-\abs{\nabla_{\g_1}\ek_1}_{\g_1}^2}.
\end{align}
Taking into account  (\ref{3.8}) and (\ref{3.7}), the right-hand side of (\ref{3.14}) becomes
\begin{equation}\label{3.15}
z_{1,h} (t,x)=-e^{i(\ek_1(x)-t)/h}(\p_t^2-\Delta_{\g_1})a_1(t,x)\equiv-e^{i(\ek_1(x)-t)/h}z(t,x).
\end{equation}
Since $a_1\in\mathcal{A}(Q)$ and satisfies (\ref{3.9}), we deduce that $z\in H^1_0(0,T^*;L^2(\M))$. Furthermore, there is a constant $C>0$ does not depend on $h$ and $\varepsilon$, such that
\begin{equation}\label{3.16}
\norm{z}_{L^2(Q^*)}+\norm{\p_tz}_{L^2(Q^*)}\leq C\nr(a_1).
\end{equation}
By Lemma \ref{L3.1}, we find
\begin{equation}\label{3.17}
r_{1,h}\in \mathcal{C}^2([0,T^*];L^2(\M))\cap\mathcal{C}^1([0,T^*];H^1_0(\M))\cap\mathcal{C}([0,T^*];H^2(\M)).
\end{equation}
Further, the function
$$
\breve{r}_{1,h}(t,x)=\int_{t}^{T^*}r_h(s,x)ds,
$$
solves the mixed hyperbolic problem (\ref{3.12}) with the right side
$$
\breve{z}_{1,h}(t,x):=\int_{t}^{T^*}z_{1,h} (x,s)ds=-ih\int_t^{T^*}z(x,s)\p_s\para{e^{i(\ek_1(x)-s)/h}}ds.
$$
Integrating by part with respect to $s$, we conclude from (\ref{3.16}), that
$$
\norm{\breve{z}_{1,h}}_{L^2(Q^*)}\leq Ch\nr(a_1).
$$
and by (\ref{3.4}), we get
\begin{eqnarray}\label{3.18}
\norm{r_{1,h}(t,\cdot)}_{L^2(\M)}=\norm{\p_t\breve{r}_{1,h}(t,\cdot)}_{L^2(\M)}\leq
Ch\nr(a_1).
\end{eqnarray}
Since $\norm{z_{1,h}}_{L^2(Q^*)}+h\norm{\p_tz_{1,h} }_{L^2(Q^*)}\leq C\nr(a_1)$, by using again the energy estimates for the problem (\ref{3.12}), obtain
\begin{eqnarray}\label{3.19}
\norm{\p_tr_{1,h}(t,\cdot)}_{L^2(\M)}+\norm{\nabla r_{1,h} (t,\cdot)}_{L^2(\M)}\leq C\nr(a_1).
\end{eqnarray}
and by (\ref{3.5}), we have
\begin{eqnarray}\label{3.20}
\norm{\p_t^2 r_{1,h}(t,\cdot)}_{L^2(\M)}+\norm{\p_t r_{1,h}(t,\cdot)}_{H^1(\M)}+\norm{r_{1,h}(t,\cdot)}_{H^2(\M)}\leq Ch^{-1}\nr(a_1).
\end{eqnarray}
Collecting (\ref{3.18})-(\ref{3.19}) and (\ref{3.20}) we get (\ref{3.11}).
The proof is complete.
\end{proof}

We will now construct the phase function $\ek_1$ solution to the eikonal equation (\ref{3.7}) and the amplitude $a_1$ solution to the transport equation (\ref{3.8}). As mentioned above the Riemannian manifold $(\M,\g)$ is simple. Then the eikonal equation (\ref{3.7}) can be solved globally on $\M$. To see this, we pick $y\in \Gamma_1:=\p \M_1$. Denote points in $\M_1$ by $(r,\theta)$ where $(r,\theta)$ are polar normal coordinates in $\M_1$ with center $y$. That is $x=\exp_{y}(r\theta)$ where $r>0$ and $\theta\in S_{y}\M_1=\set{\xi\in T_{y}\M_1,\,\,\abs{\xi}=1}$. In these coordinates (which depend on the choice of $y$) the metric takes the form $\widetilde{\g}_1(r,\theta)=\dd r^2+\check{\g}_1(r,\theta)$,  where $\check{\g}_0(r,\theta)$ is a smooth positive definite metric. For any function $u$ compactly supported in $\M$, we set for $r>0$ and $\theta\in S_y\M_1$,
$$
\widetilde{u}(r,\theta)=u(\exp_{y}(r\theta)),
$$
where we have extended $u$ by $0$ outside $\M$ and we use this notation to indicate the representation in the polar normal coordinates. An explicit solution to the eikonal equation (\ref{3.7}) is the geodesic distance function to $y \in \Gamma_1$
\begin{equation}\label{3.21}
\ek_1(x)=d_{\g_1}(x,y).
\end{equation}
By the simplicity (smallness) assumption, since $y\in \overline{\M}_1\backslash\overline{\M}$, we have $\ek_1\in\mathcal{C}^\infty(\M)$ and
\begin{equation}\label{3.22}
\widetilde{\ek}_1(r,\theta)=r=d_{\g_1}(x,y).
\end{equation}
Moreover, we have (see \cite{B})
\begin{equation}\label{3.23}
\nabla\ek_1(x)=\dot{\gamma}_{y,\theta}(r),\quad r=d_{\g_1}(x,y)
\end{equation}
where $\gamma_{y,\theta}$ is the unique geodesic connecting $y$ to $x$.\\
The next step is to solve the transport equation (\ref{3.8}). Let $\alpha_{\g_1}=\alpha_{\g_1}(r,\theta)$ be the square of the volume element in geodesic polar coordinates. The transport equation (\ref{3.8}) becomes (see \cite{BF})
\begin{equation}\label{3.26}
\p_t \widetilde{a}_1+\p_r \widetilde{a}_1+\frac{1}{4}\widetilde{a}_1\alpha_{\g_1}^{-1}\p_r \alpha_{\g_1}=0.
\end{equation}
Let $\kappa_1\in\mathcal{C}_0^\infty(\R)$ and $b\in H^2(\p_+S\M)$. Let us define $\widetilde{a}_1$ by
\begin{equation}\label{3.27}
\widetilde{a}_1(t,r,\theta)=\alpha_{\g_1}^{-1/4}\kappa_1(t-r)b(y,\theta).
\end{equation}
Then $\widetilde{a}_1$ is a solution of the transport equation (\ref{3.8}).
Now if we assume that $\mathrm{supp}\kappa_1 \subset (0,\delta)$,  then for any $x=\exp_y(r\theta)\in \M$, it is easy to see that $\widetilde{a}_1(t,r,\theta)=0$ if $t\leq \delta$ and $t\geq T^*-\delta$ for some $T^*>\mathop{\rm Diam} \M+2\delta$.
\subsection{Solutions for the forward wave equation}
Let $\ek_1$, $\ek_2$ be two phase functions solving the eikonal equation with respect to the metrics $\g_1$ and $\g_2$ respectively.
\begin{equation}\label{3.31}
\abs{\nabla_{\g_1}\ek_1}^2_{\g_1}=\sum_{j,k=1}^n\g_1^{jk}\p_j\ek_1 \p_k\ek_1=1,\quad 
\abs{\nabla_{\g_2}\ek_2}^2_{\g_2}=\sum_{j,k=1}^n \g_2^{jk}\p_j\ek_2\p_k\ek_2=1,
\quad\textrm{ on } \M.
\end{equation}
Let $a_2\in \mathcal{A}(Q)$ solve the transport equation in $\R\times\M$ with respect to the metric $\g_1$ and the phase function $\ek_1$
\begin{equation}\label{3.32}
\mathcal{L}_{\g_1,\ek_1}a_2(t,x):=\p_t a_2+\sum_{j,k=1}^n\g_1^{jk}\p_j\ek_1\p_k a_2+\frac{a_2}{2}\Delta_{\g_1}\ek_1=0,
\end{equation}
and we denote the symmetric $2$-tensor $\mathfrak{s}=(\mathrm{s}_{jk})_{jk}$ defined by
\begin{equation}\label{3.33}
\mathrm{s}_{jk}=\g_2^{jk}-\g_1^{jk},\quad j,k=1,\ldots,n.
\end{equation}
Let $a_3\in \mathcal{A}(Q)$ solve the following transport equation in $\R\times\M$ with respect to the metric~$\g_2$
\begin{equation}\label{3.34}
\mathcal{L}_{\g_2,\ek_2}a_3(t,x;h):=\p_t a_3+\sum_{j,k=1}^n \g_2^{jk} \p_j \ek_2\p_k a_3+\frac{a_3}{2}\Delta_{\g_2}\ek_2= a_2(t,x)m(x,h),
\end{equation}
where
\begin{equation}\label{3.35}
m(x,h)=-\frac{i}{2}e^{i(\ek_1-\ek_2)/h}\sum_{j,k=1}^n \mathrm{s}_{jk} \p_j\ek_1\p_k \ek_1,
\end{equation}
and which satisfies the bound
\begin{equation}\label{3.36}
\nr(a_3)\leq Ch^{-2} \norm{\mathfrak{s}}_{\mathcal{C}^2} \nr(a_2).
\end{equation}
Let us now explain how to construct a solution $a_3$ satisfying (\ref{3.34}) and (\ref{3.36}). To solve the transport equation (\ref{3.34})
with (\ref{3.36}) it is enough to take, in the geodesic polar coordinates $(r,\theta)$ (with respect to the metric $\g_2$)
\begin{equation}\label{3.37}
\widetilde{a}_3(t,r,\theta;h)=\alpha_{\g_2}^{-1/4}(r,\theta)\int_0^r\alpha_{\g_2}^{1/4}(s,\theta)\widetilde{a}_2(s-r+t,s,\theta)
\widetilde{m}(s,\theta,h) \, \dd s,
\end{equation}
where $\alpha_{\g_2}(r,\theta)$ denotes the square of the volume element in geodesic polar coordinates with respect to the metric $\g_2$.
Using that $\norm{m(\cdot,h)}_{\mathcal{C}^2}\leq Ch^{-2}  \norm{\mathfrak{s}}_{\mathcal{C}^2} $ and (\ref{3.37})
we obtain (\ref{3.36}).
\begin{lemma}\label{L3.3} Let  $\g_\ell\in \mathcal{C}^2(\M)$ such that $\norm{\g_\ell}_{\mathcal{C}^k(\M)}\leq M_0$, $\ell=1,2$. Let $a_2,a_3\in\mathcal{A}(Q)$ solve respectively (\ref{3.32})-(\ref{3.34}) and $\ek_1,\ek_2$ satisfy (\ref{3.31}) and let $T>\textrm{Diam}\,\M$. Then for all $h>0$ small enough, there exists a solution 
$$
u_2(t,x;h)\in \mathcal{C}^2([0,T];L^2(\M))\cap\mathcal{C}^1([0,T];H^1(\M))\cap\mathcal{C}([0,T];H^2(\M))
$$ of the wave equation
\begin{equation*}
(\p^2_t-\Delta_{\g_2})u_2=0,\quad \textrm{in}\quad Q:=(0,T)\times\M,
\end{equation*}
with the initial condition
\begin{equation*}
u_2(0,x)=\p_tu_2(0,x)=0,\quad \textrm{in}\quad \M,
\end{equation*}
 of the form
\begin{equation}\label{3.38}
u_2(t,x;h)=ha_2(t,x)e^{i(\ek_1(x)-t)/h}+ a_3(t,x)e^{i(\ek_2(x)-t)/h}+r_{2,h}(t,x),
\end{equation}
the remainder $r_{2,h}(t,x)$ is such that
\begin{align*}
r_{2,h}(t,x)&=0,\quad (t,x)\in \Sigma , \\
r_{2,h}(0,x)=\p_tr_{2,h}(0,x)&=0,\quad x\in \M.
\end{align*}
Furthermore, there exist $C>0$, $h_0>0$ such that, for all $h \leq h_0$ the following estimates hold true.
\begin{equation}\label{3.39}
\sum_{k=0}^2 \sum_{j=0}^{k} h^{k-1} \|\p_t^{j} r_{2,h}(t,\cdot)\|_{H^{k-j}(\M)}\leq C(h+h^{-2}\norm{\mathfrak{s}}_{\mathcal{C}^2})\nr(a_2).
\end{equation}
The constant $C$ depends only on $T$, $\M$ and $M_0$ (that is $C$ does
not depend on $a$ and $h$ and $\varepsilon$). 
\end{lemma}
\begin{proof}
We set
\begin{align}\label{3.40}
z_{2,h}(t,x)=&-\para{\partial_t^2-\Delta_{\g_2}}\big(h a_2(t,x)e^{i(\ek_1- t)/h} + a_3(t,x,h)e^{i(\ek_2- t)/h}\big).
\end{align}
To prove our Lemma it would be enough to show that if $r_{2,h}$ solves
\begin{equation}\label{3.41}
\para{\partial_t^2-\Delta_{\g_2}}r_{2,h}=z_{2,h}(t,x)
\end{equation}
with final and boundary conditions
\begin{equation}\label{3.42}
r_{2,h}(0,x)=\p_tr_{2,h}(0,x)=0, \textrm{ in }\M,\quad \textrm{and}\quad r_{2,h}(t,x)=0 \textrm{ on }\,\Sigma,
\end{equation}
then the estimates (\ref{3.39}) holds. We have
\begin{align}\label{3.43}
-z_{2,h}(t,x)=&h e^{i(\ek_1-t)/h}\para{\p_t^2-\Delta_{\g_2}}a_2+2ie^{i(\ek_1-t)/h}\Big(\p_ta_2+
\sum_{j,k=1}^n \g_2^{jk}\p_j\ek_1\p_k a_2+\frac{a_2}{2}\Delta_{\g_2}\ek_1\Big)\cr
&\quad +h^{-1}a_2e^{i(\ek_1-t)/h}\Big(1-\sum_{j,k=1}^n\g_2^{jk}\p_j\ek_1\p_k\ek_1\Big)
+e^{i(\ek_2-t)/h}\para{\p_t^2-\Delta_{\g_2}}a_3\cr
&\quad +2ih^{-1} e^{i(\ek_2- t)/h}\mathcal{L}_{\g_2,\ek_2}a_3(t,x)
+h^{-2}a_3 e^{i(\ek_2-t)/h}\Big(1-\abs{\nabla_{\g_2}\ek_2}^2_{\g_2}\Big).
\end{align}
Taking into account (\ref{3.31}) and (\ref{3.31}), the right-hand side of (\ref{3.43}) becomes
\begin{align}\label{3.44}
-z_{2,h}(t,x)=&e^{i(\ek_1-t)/h}\Big[\big(\p_t^2-\Delta_{\g_2}\big) a_2+2i\Big(\sum_{j,k=1}^n \mathrm{s}_{jk}\p_j\ek_1\p_k a_2+\frac{a_2}{2}\big(\Delta_{\g_2}\ek_1-\Delta_{\g_1}\ek_1\Big)\Big]\cr
&\quad +2ih^{-1}e^{i(\ek_2- t)/h}\Big[\mathcal{L}_{\g_2,\ek_2}a_3(t,x)
-a_2(t,x)m(x,h)\Big] \cr
&\quad +e^{i(\ek_2-t)/h}\para{\p_t^2-\Delta_{\g_2}}a_3.
\end{align}
By (\ref{3.34}) we get
\begin{multline}\label{3.45}
-z_{2,h}(t,x)=he^{i(\ek_1-t)/h}\big(\p_t^2-\Delta_{\g_2}\big) a_2\cr
 +2ie^{i\lambda(\ek_1- t)}\Big[\sum_{j,k=1}^n \mathrm{s}_{jk}\p_j\ek_1\p_k a_2+\frac{a_2}{2}\big(\Delta_{\g_2}\ek_1-\Delta_{\g_1}\ek_1\big)\Big]+
e^{i(\ek_2-t)/h}\para{\p_t^2-\Delta_{\g_2}}a_3\cr
\equiv \Big[e^{i\ek_1/h}(hk_0+k_1)+e^{i\ek_2/h}k_2\Big]e^{-it/h}=z_{0,h}(t,x)e^{-it/h},
\end{multline}
where
$$
k_0=\big(\p_t^2-\Delta_{\g_2}\big) a_2,\quad k_1=\sum_{j,k=1}^n \mathrm{s}_{jk}\p_j\ek_1\p_k a_2
+\frac{a_2}{2}\big(\Delta_{\g_2}\ek_1-\Delta_{\g_1}\ek_1\big),\quad k_2=\para{\p_t^2-\Delta_{\g_2}}a_3.
$$
Since $k_j\in H_0^1(0,T;L^2(\M))$, $j=0,1,2$, by Lemma \ref{L3.1}, we deduce that
\begin{equation}\label{3.46}
r_{2,h}\in \mathcal{C}^2([0,T];L^2(\M))\cap\mathcal{C}^1([0,T]; H^1_0(\M))\cap\mathcal{C}([0,T];H^2(\M)).
\end{equation}
Further, the function
\begin{equation}\label{3.47}
\breve{r}_{2,h}=\int_0^{t} r_{2,h}(s,x)ds
\end{equation}
solves the mixed problem (\ref{3.41})-(\ref{3.42}), with the right hand side
\begin{equation}\label{3.48}
\breve{z}_{2,h}(t,x)=\int_0^{t} z_{2,h}(s,x)ds=-ih\int_0^{t} z_{0,h}(s,x)\p_se^{-is/h}ds.
\end{equation} 
Integrating by part with respect $s$, we obtain 
\begin{equation}\label{3.49}
\norm{\breve{z}_{2,h}}_{L^2(Q)}\leq C\para{h^2+h^{-1}\norm{\mathfrak{s}}_{\mathcal{C}^2}}\nr(a_2),
\end{equation}
and by the energy estimate (\ref{3.4}) for $\breve{r}_{2,h}$, we get
\begin{equation}\label{3.50}
\norm{r_{2,h}(t,\cdot)}_{L^2(\M)}=\norm{\p_t\breve{r}_{2,h}(t,\cdot)}_{L^2(\M)}\leq C\para{h^2+h^{-1}\norm{\mathfrak{s}}_{\mathcal{C}^2}}\nr(a_2).
\end{equation}
Moreover, we have
\begin{eqnarray}\label{3.51}
\norm{z_{2,h}}_{L^2(Q)}\leq C\para{h+ h^{-2}\norm{\mathfrak{s}}_{\mathcal{C}^2}}\nr(a_2),
\end{eqnarray}
and by using again the energy estimates (\ref{3.4}) for the problem (\ref{3.41})-(\ref{3.42}), we obtain
\begin{eqnarray}\label{3.52}
\norm{\nabla r_{2,h} (t,\cdot)}_{L^2(\M)}+\norm{\p_t r_{2,h}  (t,\cdot)}_{L^2(\M)}\leq C\para{h+h^{-2}\norm{\mathfrak{s}}_{\mathcal{C}^2}}\nr(a_2),
\end{eqnarray}
and by (\ref{3.5}), we get 
\begin{eqnarray}\label{3.53}
\norm{\p_t^2 r_{2,h}(t,\cdot)}_{L^2(\M)}+\norm{\p_t r_{2,h}(t,\cdot)}_{H^1(\M)}+\norm{ r_{2,h}(t,\cdot)}_{H^2(\M)}\leq C(1+h^{-3}\norm{\mathfrak{s}}_{\mathcal{C}^2})\nr(a_2).
\end{eqnarray}
This completes the proof.
\end{proof}
\subsection{Integral identity}

This section is devoted to the proof of some integral identity. We use the following notations; let
$\g_1,\g_2\in\mathcal{C}^k(\M)$, we recall that $\mathrm{s}_{jk}(x)=(\g_2^{jk}-\g_1^{jk})(x)$, and we denote
\begin{equation}\label{3.54}
 \alpha(x)=\frac{\sqrt{\abs{\g_2}}}{\sqrt{\abs{\g_1}}},\quad \beta(x)=\alpha(x)-1,\quad \varrho(x)=\alpha(x)\sum_{j,k=1}^n \mathrm{s}_{jk}(x)\p_j\ek_1\p_k \ek_1,
\end{equation}
and $\mathfrak{s}'=(\mathrm{s}'_{jk})_{jk}$, with
\begin{equation}\label{3.55}
\mathrm{s}'_{jk}(x)=\left( \alpha(x)\g_2^{jk}-\g_1^{jk}\right).
\end{equation}
Then the following holds
\begin{equation}\label{3.56}
\norm{\varrho}_{L^2(\M)}\leq C\norm{\mathfrak{s}}_{L^2(\M)},\quad 
 \norm{\beta}_{\mathcal{C}^0(\M)}\leq C\norm{\mathfrak{s}}_{\mathcal{C}^0(\M)}.
\end{equation}
Indeed, for the last inequality, since $\bar{\M}$ is compact, there exist $m_0>0$ such that $\abs{\g_j}\geq m_0$, $j=1,2$, then
$$
\abs{\beta(x)}\leq\frac{1}{2m_0}\Big|\abs{\g_2}-\abs{\g_1}\Big|\leq C\norm{\g_2-\g_1}_{\mathcal{C}^0(\M)},
$$
where $\abs{\g}$, denotes the determinant of $\g$, and
$$
C=\sup_{\rho\in[0,1]}\norm{D(\det)\para{(1-\rho)\g_2+\rho\g_1}}.
$$
Let now $\g_1,\g_2\in \mathcal{C}^k(\M)$ two metrics tensor such that $\g_1=\g_2$ in $\mathcal{O}$ and $T>T^*>\textrm{Diam}\,\M$. Let $u_1$, $u_2$ such that
$$
u_1\in \mathcal{C}^2([0,T^*];L^2(\M))\cap\mathcal{C}^1([0,T^*];H^1(\M))\cap\mathcal{C}([0,T^*];H^2(\M))
$$
$$
u_2\in \mathcal{C}^2([0,T];L^2(\M))\cap\mathcal{C}^1([0,T];H^1(\M))\cap\mathcal{C}([0,T];H^2(\M))
$$
and solve respectively the following boundary problems in $Q^*=(0,T^*)\times\M$ and $Q=(0,T)\times\M$
\begin{equation}\label{3.57}
\left\{
\begin{array}{llll}
(\partial_t^2-\Delta_{\g_1})u_1=0,  & \textrm{in }\; Q^*,\cr
u_1(T^*,\cdot )=\p_tu_1(T^*,\cdot )=0, & \textrm{in }\; \M,
\end{array}
\right.
\quad ; \quad 
\left\{
\begin{array}{llll}
(\partial^2_t-\Delta_{\g_2})u_2=0,  & \textrm{in }\; Q,\cr
u_2(0,\cdot )=\p_tu_2(0,\cdot )=0, & \textrm{in }\; \M.
\end{array}
\right.
\end{equation}
Let $u$ the unique solution of the following initial boundary value problem
\begin{equation}\label{3.58}
\left\{
\begin{array}{llll}
(\partial_t^2-\Delta_{\g_1})u=0,  & \textrm{in }\; Q,\cr
u(0,\cdot )=\p_tu(0,\cdot )=0, & \textrm{in }\; \M,\cr
u=u_2 & \textrm{on }\; \Sigma.
\end{array}
\right.
\end{equation}
We denote $w=u-u_2$. Then $w$ solves the initial boundary value problem
\begin{equation}\label{3.59}
\left\{
\begin{array}{llll}
(\partial_t^2-\Delta_{\g_1})w=-(\p_t^2-\Delta_{\g_1})u_2(t,x),  & \textrm{in }\; Q,\cr
w(0,\cdot )=\p_tw(0,\cdot )=0, & \textrm{in }\; \M,\cr
w=0 & \textrm{on }\; \Sigma,
\end{array}
\right.
\end{equation}
Let  $\varkappa\in\mathcal{C}_0^\infty(\M)$ with $\varkappa=1$ in $\M\backslash\mathcal{O}$, and consider $w_0=\varkappa w$. Then $w_0$ solves the following initial boundary value problem
\begin{equation}\label{3.60}
\left\{
\begin{array}{llll}
(\partial_t^2-\Delta_{\g_1})w_0=-(\p_t^2-\Delta_{\g_1})u_2(t,x)-[\Delta_{\g_1},\varkappa]w,  & \textrm{in }\; Q,\cr
w_0(0,\cdot )=\p_tw_0(0,\cdot )=0, & \textrm{in }\; \M,\cr
w_0=0 & \textrm{on }\; \Sigma,
\end{array}
\right.
\end{equation}
where we have used that $(\p_t^2-\Delta_{\g_1})u_2=(\p_t^2-\Delta_{\g_2})u_2=0$ in $\mathcal{O}$. We multiply both hand sides of the first equation (\ref{3.60}) by $\overline{u}_1$, integrate by parts in time and use Green's formula (\ref{2.5}) to get
\begin{align*}
0&=\int_{Q^*} \overline{u}_1(\p^2_t-\Delta_{\g_1}) w_0\dv_{\g_1} \, \dd t  \cr
&=-\int_{Q^*} \overline{u}_1(\p^2_t-\Delta_{\g_1})u_2\dv_{\g_1} \, \dd t-\int_{Q^*} \overline{u}_1[\Delta_{\g_1},\varkappa]w\dv_{\g_1} \, \dd t \cr
&=-\int_{Q^*} \overline{u}_1 [\Delta_{\g_1},\varkappa]w\dv_{\g_1} \, \dd t +\int_{Q^*} \p_tu_2\p_t\overline{u}_1 \dv_{\g_1} \, \dd t\cr
&\quad+\int_{Q^*} \overline{u}_1\Delta_{\g_2}u_2\dv_{\g_2}dt +\int_{Q^*} \Big(\sum_{j,k=1}^n \mathrm{s}'_{jk}\p_j u_2\p_k \overline{u}_1\Big)\dv_{\g_1} \, \dd t,
\end{align*}
and after using a second time Green's formula  (\ref{2.5}), we end up with the following integral identity
\begin{align}\label{3.61}
\int_{Q^*} u_1 [\Delta_{\g_1},\varkappa]\overline{w}\dv_{\g_1} \, \dd t &=-\int_{Q^*} \beta(x) \p_tu_1\p_t\overline{u}_2 \dv_{\g_1} \, \dd t+\int_{Q^*} \Big(\sum_{j,k=1}^n \mathrm{s}'_{jk}\p_j u_1\p_k \overline{u}_2\Big)\dv_{\g_1} \, \dd t.
\end{align}
Next, taking inspiration to the analysis carried out in \cite{Bell-Benfraj} we obtain a stability estimate in the unique continuation of solution of the  wave equation from lateral boundary data on an arbitrary non-empty relatively open subset $\Gamma^\natural$ of $\Gamma$. We shall use the following notations. Let $\mathcal{O}\subset\M$ be a given smooth neighborhood of the boundary $\Gamma$, we consider then three open subset $\mathcal{O}_j$, $j=1,2,3$ of $\mathcal{O}$, such that
\begin{equation}\label{3.62}
\bar{\mathcal{O}}_{j+1}\subset \mathcal{O}_j,\quad \Gamma\subset\p\mathcal{O}_j,
\end{equation} 
and we select the cut-off function $\varkappa$ as in (\ref{3.61}) such that
\begin{equation}\label{3.63}
\varkappa(x)=0,\quad x\in\mathcal{O}_3,\quad \varkappa(x)=1, \quad x\in\bar{\M}\backslash\mathcal{O}_2.
\end{equation}
We have the following Lemma.
\begin{lemma}\label{L3.4}
Let $T$ be sufficiently large. Then there exist positive constants $C,\; T_*\in (\textrm{Diam}\,\Omega,T), \; \mu$ and $\gamma_*$ such that the estimate 
\begin{equation}\label{3.64}
\|w\|_{L^2((0,T^*)\times(\mathcal{O}_2\backslash\mathcal{O}_3))}\leq C {\gamma^{-\frac{1}{2}}}\|w\|_{H^1(Q)}+e^{\mu \gamma}(\|(\p_t^2-\Delta_{\g_1})w\|_{L^2((0,T)\times\mathcal{O})}+\|\partial_\nu w\|_{L^2(\Sigma^\natural)}).
\end{equation}
holds for any $\gamma>\gamma_*$ and any $w \in H^2(Q)$ such that $w=0$ on $\Sigma$.
\end{lemma}
The key idea of the proof of Lemma \ref{L3.4} is to combine the analysis carried out in \cite{BCY}, which is based on a Carleman estimate specifically designed for the system under consideration, with the Fourier-Bros-Iagolnitzer (FBI) transformation. Indeed we take advantage of the fact that the FBI transform of the time derivative of the solution $w$ satisfies elliptic equation in the vicinity of the boundary in order to apply a Carleman elliptic estimate where no geometric condition is imposed on the control domain.
\section{Recovery of symmetric $2$-tensor}
\setcounter{equation}{0}
\begin{lemma}\label{L4.1}
There exists $C>0$ and $T^*\in (\textrm{Diam}\,\M,T)$ such that for any $a_1$, $a_2\in\mathcal{A}(Q)$ satisfying the transport equation \eqref{3.32} with \eqref{3.9} the
following estimate holds true
\begin{multline}\label{4.1}
\abs{\int_0^{T^*}\!\!\!\int_\M \varrho(x)(a_1\overline{a}_2)(t,x)\dv_{\g_1} \, \dd t}\leq C\Big[\norm{\mathfrak{s}}_{\mathcal{C}^2}
\para{h+h^{-3}\norm{\mathfrak{s}}_{\mathcal{C}^2}}\cr
+ \gamma^{-\frac{1}{2}}(h+h^{-2}\norm{\mathfrak{s}}_{\mathcal{C}^2})+e^{\mu \gamma}h^{-3} \| \Lambda_{\g_1}^{\natural}-\Lambda_{\g_2}^{\natural}\rVert \Big]\nr(a_1)\nr(a_2).
\end{multline}
for any sufficiently large $\gamma$ and sufficiently small $h$.
\end{lemma}
\begin{proof}
Let $T^*$ satisfying Lemma \ref{L3.4}. Following Lemma \ref{L3.3} let $u_2$ be a solution to the wave equation $(\p_t^2-\Delta_{\g_2})u_2=0$ in $Q$, with the initial data $u_2(0,\cdot)=\p_tu_2(0,\cdot)= 0$ in $\M$ of the form
$$
\overline{u}_2(t,x)=h\overline{a}_2( t,x)e^{-i(\ek_1- t)/h}+\overline{a}_3
( t,x;h)e^{-i(\ek_2- t)/h}+\overline{r}_{2,h}(t,x),
$$
where $r_{2,h}$ satisfies (\ref{3.39}) and $a_3$ satisfies (\ref{3.36}).
Thanks to Lemma \ref{L3.2} let $u_1$ be a solution to the wave equation $(\p_t^2-\Delta_{\g_1})u_1=0$ in $Q$, with the final data $u_1(T^*,\cdot)=\p_t u_1(T^*,\cdot)=0$ in $\M$ of the form
$$
u_1(t,x)=a_1( t,x)e^{i(\ek_1-t)/h}+r_{1,h}(t,x),
$$
where $r_{1,h}$ satisfies (\ref{3.11}). Then
\begin{align}\label{4.2}
\p_t\overline{u}_2(t,x)&= h\p_t\overline{a}_2( t,x)e^{-i(\ek_1- t)/h}+i
\overline{a}_2( t,x)e^{-i(\ek_1- t)/h}\cr
& \quad +\p_t\overline{a}_3( t,x;h)e^{-i(\ek_2- t)/h}+ih\overline{a}_3
( t,x,h)e^{-i(\ek_2- t)/h}
+\p_t\overline{r}_{2,h}(t,x)\cr
\p_tu_1(t,x)&=\p_ta_1(t,x)e^{i(\ek_1-t)/h}-ih^{-1}a_1(t,x)e^{i(\ek_1-t)/h}+\p_tr_{1,h}.
\end{align}
Let us compute the first term in the right hand side of (\ref{3.61}). We have
\begin{align*}
\nonumber \int_0^{T^*}\!\!\!\int_\M \beta(x) &\p_tu_1 \p_t\overline{u}_2 \dv_{\g_1} \, \dd t =h^{-1}\int_0^{T^*}\!\!\!\int_\M \beta(x) a_1\overline{a}_2\dv_{\g_1} \, \dd t \cr
&+h\int_0^{T^*}\!\!\!\int_\M \beta(x)\para{\p_ta_1\p_t\overline{a}_2}\dv_{\g_1} \, \dd t
-i\int_0^{T^*}\!\!\!\int_\M \beta(x) a_1\p_t\overline{a}_2 \dv_{\g_1} \, \dd t \cr
&+h\int_0^{T^*}\!\!\!\int_\M \beta(x)\p_t\overline{a}_2\p_tr_{1,h} e^{-i(\ek_1-t)/h}\dv_{\g_1} \, \dd t
+i\int_0^{T^*}\!\!\!\int_\M \beta(x)\para{\p_ta_1\overline{a}_2}\dv_{\g_1} \, \dd t  \cr
&+i \int_0^{T^*}\!\!\!\int_\M \beta(x)\overline{a}_2\p_tr_{1,h} e^{-i(\ek_1-t)/h}\dv_{\g_1} \, \dd t+\int_0^{T^*}\!\!\!\int_\M
\beta(x)\p_t\overline{a}_3\p_ta_{1}e^{i(\ek_1-\ek_2)/h}\dv_{\g_1} \, \dd t  \cr
&-ih^{-1}\int_0^{T^*}\!\!\!\int_\M \beta(x)\p_t\overline{a}_3a_{1}e^{i(\ek_1-\ek_2)/h}\dv_{\g_1} \,\dd t+\!\int_0^{T^*}\!\!\!\int_\M
\beta(x)\p_tr_{1,h}\p_t\overline{a}_3e^{-i(\ek_2-t)/h}\dv_{\g_1} \dd t  \cr
&+ih^{-1}\int_0^{T^*}\!\!\!\int_\M \beta(x)\overline{a}_3\p_ta_{1}e^{i(\ek_1-\ek_2)/h}\dv_{\g_1} \, \dd t+h^{-2}\int_0^{T^*}\!\!\!\int_\M
\beta(x)a_1\overline{a}_3e^{i(\ek_1-\ek_2)/h}\dv_{\g_1} \, \dd t  \cr
&+ih^{-1}\int_0^{T^*}\!\!\!\int_\M \beta(x)\p_tr_{1,h}\overline{a}_3e^{-i(\ek_2-t)/h}\dv_{\g_1} \, \dd t+\!\int_0^{T^*}\!\!\!\int_\M
\beta(x)\p_ta_1\p_t\overline{r}_{2,h}e^{i(\ek_1-t)/h} \dv_{\g_1} \dd t  \cr
&-ih^{-1}\int_0^{T^*}\!\!\!\int_\M \beta(x)a_1\p_t\overline{r}_{2,h}e^{i(\ek_1-t)/h} \dv_{\g_1} \, \dd t+\int_0^{T^*}\!\!\!\int_\M
\beta(x)\p_tr_{1,h}\p_t\overline{r}_{2,h}\dv_{\g_1} \, \dd t.
\end{align*}
Thus, we have from (\ref{3.11}), (\ref{3.36}), (\ref{3.39}) and (\ref{3.56}) the following identity
\begin{equation}\label{4.3}
\int_0^{T^*}\!\!\!\int_\M \beta(x)\p_tu_1\p_t\overline{u}_2\dv_{\g_1} \, \dd t=
h^{-1}\int_0^{T^*}\!\!\!\int_\M \beta(x)(a_1\overline{a}_2)(t,x)\dv_{\g_1} \, \dd t+\mathcal{J}_1(h),
\end{equation}
where
\begin{equation}\label{4.4}
\abs{\mathcal{J}_1(h)}\leq\norm{\mathfrak{s}}_{\mathcal{C}^2}\para{1+h^{-4}  \|\mathfrak{s}\|_{\mathcal{C}^2(\M)}}\nr(a_2)\nr(a_1).
\end{equation}
On the other hand, we have
\begin{align}
\nonumber
\p_j u_1 &= \p_j a_1 e^{i(\ek_1-t)/h}+
ih^{-1} \p_j \ek_1 a_1 e^{i(\ek_1- t)/h}+\p_j r_{1,h} \\
\label{4.5}
\p_k \overline{u}_2 &=h\p_k\overline{a}_2 e^{-i(\ek_1- t)/h}-i\overline{a}_2 \p_k\ek_1e^{-i(\ek_1- t)/h} \cr
&\quad -ih^{-1}\overline{a}_3\p_k\ek_2e^{-i(\ek_2- t)/h}+\p_k\overline{a}_3 e^{-i(\ek_2- t)/h}+\p_k\overline{r}_{2,h}.
\end{align}
and the second term in the right side of (\ref{3.61}) becomes
\begin{multline}\label{4.6}
\int_0^{T^*}\!\!\!\int_\M\seq{\mathrm{s}'(x)\nabla u_1(t,x),\nabla \overline{u}_2(t,x)}_0 \dv_{\g_1} dt
\\ =h^{-1}\int_0^{T^*}\!\!\!\int_\M \seq{\mathrm{s}'(x)\nabla \ek_1,\nabla \ek_1}_0  (a_1\overline{a}_2)(t,x)\dv_{\g_1} \, \dd t+\mathcal{J}_2(h)+\mathcal{J}_3(h)
\end{multline}
with
\begin{align*}
\mathcal{J}_2(h)&=h\int_0^{T^*}\!\!\!\int_\M\seq{\mathrm{s}'(x) \nabla a_1,\nabla\overline{a}_2}_0\dv_{\g_1}  \, \dd t
-i\int_0^{T^*}\!\!\!\int_\M\overline{a}_2\seq{\mathrm{s}'(x)\nabla a_1,\nabla\ek_1(x)}_0\dv_{\g_1} \, \dd t\cr
&\quad+i\int_0^{T^*}\!\!\!\int_\M  a_1\seq{\mathrm{s}'(x) \nabla\overline{a}_2, \nabla\ek_1}_0\dv_{\g_1} \dd t
+h\int_0^{T^*}\!\!\!\int_\M  e^{-i(\ek_1- t)/h}\seq{\mathrm{s}'(x) \nabla\overline{a}_2,\nabla r_{1,h}}_0\dv_{\g_1}  \dd t\cr
&\quad-i\int_0^{T^*}\!\!\!\int_\M \overline{a}_2 e^{-i(\ek_1- t)/h}\seq{\mathrm{s}'(x) \nabla r_{1,h},\nabla\ek_1}_0\dv_{\g_1}  \dd t
\end{align*}
and
\begin{align*}
&\mathcal{J}_3(h) =-ih^{-1}\int_0^{T^*}\!\!\!\int_\M \overline{a}_3 e^{i(\ek_1-\ek_2)/h}\seq{\mathrm{s}'\nabla a_1,\nabla\ek_2}_0\dv_{\g_1} \, \dd t\cr
&+\int_0^{T^*}\!\!\!\int_\M e^{i(\ek_1-\ek_2)/h}\seq{\mathrm{s}' \nabla a_1,\nabla\overline{a}_3}_0\dv_{\g_1} \, \dd t
+\int_0^{T^*}\!\!\!\int_\M  e^{i(\ek_1- t)/h}\seq{\mathrm{s}'\nabla a_1,\nabla\overline{r}_{2,h}}_0\dv_{\g_1} \, \dd t  \cr
&+\int_0^{T^*}\!\!\!\int_\M \seq{\mathrm{s}'\nabla r_{1,h}, \nabla \overline{r}_{2,h}}_0\dv_{\g_1} \, \dd t
+h^{-2}\int_0^{T^*}\!\!\!\int_\M  a_1\overline{a}_3 e^{i(\ek_1-\ek_2)/h}\seq{\mathrm{s}'\nabla\ek_1,\nabla\ek_2}_0\dv_{\g_1} \, \dd t\cr
&+ih^{-1}\int_0^{T^*}\!\!\!\int_\M  a_1e^{i(\ek_1-\ek_2)/h}\seq{\mathrm{s}'\nabla\overline{a}_3,\nabla\ek_1}_0\dv_{\g_1} \, \dd t
+ih^{-1}\int_0^{T^*}\!\!\!\int_\M  a_1e^{i(\ek_1- t)/h}\seq{\mathrm{s}' \nabla\ek_1,\nabla\overline{r}_2}_0 \dv_{\g_1} \, \dd t \cr
&-ih^{-1}\int_0^{T^*}\!\!\!\int_\M \overline{a}_3e^{-i(\ek_2- t)/h}\seq{\mathrm{s}'\nabla r_{1,h},\nabla\ek_2}_0\dv_{\g_1} \, \dd t
+\int_0^{T^*}\!\!\!\int_\M  e^{-i(\ek_2- t)/h}\seq{\mathrm{s}'\nabla r_{1,h},\nabla \overline{a}_3}_0\dv_{\g_1} \, \dd t.
\end{align*}
From (\ref{3.56}), (\ref{3.39}) and (\ref{3.11}), we have
\begin{equation}\label{4.7}
\abs{\mathcal{J}_2(h)}+\abs{\mathcal{J}_3(h)}
\leq \norm{\mathfrak{s}}_{\mathcal{C}^2}\para{1+h^{-4}  \|\mathfrak{s}\|_{\mathcal{C}^2(\M)}}\nr(a_2)\nr(a_1).
\end{equation}
Taking into account (\ref{4.3}), (\ref{4.6}) and (\ref{3.61}), we deduce that
\begin{equation}\label{4.8}
\int_{Q^*} u_1 [\Delta_{\g_1},\varkappa]\overline{w}\dv_{\g_1} \, \dd t =
h^{-1}\int_0^{T^*}\!\!\!\int_\M \varrho(x)(a_1\overline{a}_2)(t,x)\dv_{\g_1} \dd t+\mathcal{J}_1(h)+
\mathcal{J}_2(h)+\mathcal{J}_3(h),
\end{equation}
where we have used that
$$
\varrho=-\beta(x)+\sum_{j,k=1}^n\textrm{s}'_{jk}\p_j\ek_1\p_k\ek_1=\alpha(x)\sum_{j,k=1}^n\textrm{s}_{jk}\p_j\ek_1\p_k\ek_1.
$$
In view of (\ref{4.7}) and (\ref{4.4}), we obtain
\begin{align}\label{4.9}
\abs{\int_0^{T^*}\!\!\!\int_\M &\varrho(x)(a_1\overline{a}_2)(t,x)\dv_{\g_1} \dd t} \cr
&\leq C\Big(\norm{\mathfrak{s}}_{\mathcal{C}^2}
\para{h+ h^{-3} \norm{\mathfrak{s}}_{\mathcal{C}^2}}\nr(a_1)\nr(a_2)+h\|w\|_{L^2((0,T^*)\times(\mathcal{O}_2\backslash\mathcal{O}_3))}\norm{u_1}_{L^2(0,T^*;H^1(\M))}\Big)\cr
&\leq C\Big[\norm{\mathfrak{s}}_{\mathcal{C}^2}
\para{h+h^{-3}\norm{\mathfrak{s}}_{\mathcal{C}^2}}\nr(a_1)\nr(a_2)+ h\para{\gamma^{-\frac{1}{2}}\|w\|_{H^1(Q)}+e^{\mu \gamma}\|\partial_\nu w\|_{L^2(\Sigma^\natural)}}\!\nr(a_1)\Big]
\end{align}
where we have used $(\p_t^2-\Delta_{\g_1})w=0$ in $\mathcal{O}\times(0,T)$. Moreover for $w$ solution of the equation (\ref{3.59}), we have
\begin{align}\label{4.10}
\|w\|_{H^1(Q)} &\leq C \norm{(\Delta_{\g_1}-\Delta_{\g_2})u_2}_{L^2(0,T;L^2(\M))}\cr
&\leq C\norm{\mathfrak{s}}_{\mathcal{C}^2}\norm{u_2}_{L^2(0,T;H^2(\M))}
\leq C(1+ h^{-3}\norm{\mathfrak{s}}_{\mathcal{C}^2})\nr(a_2)
\end{align}
and 
\begin{align}\label{4.11}
\|\partial_\nu w\|_{L^2(\Sigma^\natural)}\leq C \| \Lambda_{\g_1}^{\natural}-\Lambda_{\g_2}^{\natural}\rVert \| f_h \rVert_{H^{1,1}(\Sigma)}.
\end{align}
Further, we have 
\begin{align*}
\|f_h\|_{H^{1,1}(\Sigma)}= \|{u_2-r_{2,h}}\|_{H^{1,1}(\Sigma)} \leq& \|u_2-r_{2,h}\|_{L^2(0,T;H^2(\M))}+\|u_2-r_{2,h}\|_{H^1(0,T;H^1(\M))}\\
\leq & C h^{-4}\nr(a_2).
\end{align*}
This, (\ref{4.9}), (\ref{4.10}) and (\ref{4.11}) yield 
\begin{multline}\label{4.12}
\abs{\int_0^T\!\!\!\int_\M \varrho(x)(a_1\overline{a}_2)( t,x)\dv_{\g_1} \dd t}
\leq C\Big[\norm{\mathfrak{s}}_{\mathcal{C}^2}
\para{h+h^{-3}\norm{\mathfrak{s}}_{\mathcal{C}^2}}+ \gamma^{-\frac{1}{2}}(h+h^{-2}\norm{\mathfrak{s}}_{\mathcal{C}^2})\cr
+h^{-3}e^{\mu \gamma} \| \Lambda_{\g_1}^{\natural}-\Lambda_{\g_2}^{\natural}\rVert \Big]\nr(a_1)\nr(a_2).
\end{multline}
This completes the proof.
\end{proof}
\begin{lemma}\label{L4.2}
There exist $C>0$ and $\mu>0$ such that for any $b\in H^2(\p_+S\M_{1})$ the following estimate
\begin{multline}\label{4.13}
\abs{\int_{\p_+S\M_1}\I_{\g_1}(\alpha\mathfrak{s})(y,\theta)b(y,\theta)\mu(y,\theta)\dss_{\g_1} (y,\theta)}\leq C\Big[\norm{\mathfrak{s}}_{\mathcal{C}^2}\para{h+h^{-3}\norm{\mathfrak{s}}_{\mathcal{C}^2}}\cr
+ \gamma^{-\frac{1}{2}}(h+ h^{-2}\norm{\mathfrak{s}}_{\mathcal{C}^2})+h^{-3}e^{\mu \gamma} \| \Lambda_{\g_1}^{\natural}-\Lambda_{\g_2}^{\natural}\rVert \Big]\norm{b(y,\cdot)}_{H^2(S^+_y\M_{1})}.
\end{multline}
holds for any  $\gamma$ sufficiently large and $h$ sufficiently small.
\end{lemma}
\begin{proof}
Following (\ref{3.27}), let $T$ sufficiently large  and take two solutions of the form
\begin{align*}
\widetilde{a}_1(t,r,\theta)&=\alpha_{\g_1}^{-1/4}\kappa(t-r)b(y,\theta), \\
\widetilde{a}_2(t,r,\theta)&=\alpha_{\g_1}^{-1/4}\kappa(t-r)\mu(y,\theta).
\end{align*}
Now we change variable in (\ref{4.1}), $x=\exp_{y}(r\theta)$, $r>0$ and
$\theta\in S_{y}\M_1$. Then
\begin{align*}
\int_0^{T^*}\!\!\int_\M &\varrho(x) a_1(t,x)a_2(t,x)\dv_{\g_1}\,\dd t\cr
&=\int_0^{T^*}\!\!\int_{S_{y}\M_1}\!\int_0^{\tau_+(y,\theta)}\widetilde{\varrho}(r,\theta)\widetilde{a}_1(t,r,\theta)
\widetilde{a}_2(t,r,\theta)\alpha_{\g_1}^{1/2} \,\dd r \, \dd \omega_y(\theta) \, \dd t \cr
&=\int_0^{T^*}\!\!\int_{S_{y}\M_1}\!\int_0^{\tau_+(y,\theta)}\widetilde{\varrho}(r,\theta)\kappa^2(t-r)b(y,\theta) \mu(y,\theta)
\, \dd r \, \dd \omega_y(\theta) \, \dd t \cr
&=\int_0^{T^*}\int_{S_{y}\M_1}\!\int_0^{\tau_+(y,\theta)}\widetilde{\varrho}(r,\theta)
\kappa^2( t-r)b(y,\theta)\mu(y,\theta) \, \dd r \, \dd \omega_y(\theta) \, \dd t.
\end{align*}
We conclude that
\begin{multline}\label{4.14}
\abs{\int_0^{T^*}\!\!\int_{S_{y}\M_1}\!\int_0^{\tau_+(y,\theta)}\widetilde{\varrho}(r,\theta)\kappa^2( t-r)b(y,\theta)\mu(y,\theta)
\, \dd r \,\dd \omega_y(\theta) \, \dd t} \leq C\Big[\norm{\mathfrak{s}}_{\mathcal{C}^2}\para{h+h^{-3}\norm{\mathfrak{s}}_{\mathcal{C}^2}}\cr
+ \gamma^{-\frac{1}{2}}(h+ h^{-2}\norm{\mathfrak{s}}_{\mathcal{C}^2})+h^{-3}e^{\mu \gamma} \| \Lambda_{\g_1}^{\natural}-\Lambda_{\g_2}^{\natural}\rVert \Big]\norm{b(y,\cdot)}_{H^2(S^+_y\M_{1})}
\end{multline}
where $S^+_y\M_{1}=\{\theta \in S_{y}\M_{1}: \langle \theta,\nu(y) \rangle_{\g}<0\}$.
Given the support properties of the function $\kappa$, the left-hand side of the inequality reads in fact
\begin{multline*}
\int_{-\infty}^{\infty}\int_{S_{y}\M_1}\!\int_0^{\tau_+(y,\theta)}\widetilde{\varrho}(r,\theta)\kappa^2( t-r)b(y,\theta)\mu(y,\theta)
\, \dd r \, \dd \omega_y(\theta) \, \dd t \cr
= \int_{S_{y}\M_1}\!\int_0^{\tau_+(y,\theta)}\widetilde{\varrho}(r,\theta)b(y,\theta)\mu(y,\theta)
\, \dd r \, \dd \omega_y(\theta).
\end{multline*}
Since $\nabla\ek_1=\dot{\gamma}_{y,\theta}(r)$, we obtain
\begin{equation}\label{4.15}
\int_0^{\tau_+(y,\theta)}\widetilde{\varrho}(r,\theta)dr=\sum_{j,k=1}^n\int_0^{\tau_+(y,\theta)}\alpha(\gamma_{y,\theta}(r))\mathrm{s}_{jk}(\gamma_{y,\theta}(r))\dot{\gamma}^j(r)\dot{\gamma}^k(r)dr=\I_{\g_1}(\alpha\mathfrak{s})(y,\theta).
\end{equation}
Integrating with respect to $y\in \p \M_1$ in \eqref{4.14} we obtain
\begin{multline*}
\abs{\int_{\p_+S\M_1}\I_{\g_1}(\alpha\mathfrak{s})(y,\theta)b(y,\theta)\mu(y,\theta)\dss_{\g_1} (y,\theta)}
\leq C\Big[\norm{\mathfrak{s}}_{\mathcal{C}^2}\para{h+h^{-3}\norm{\mathfrak{s}}_{\mathcal{C}^2}}\cr
+ \gamma^{-\frac{1}{2}}(h+ h^{-2}\norm{\mathfrak{s}}_{\mathcal{C}^2})+h^{-3}e^{\mu \gamma} \| \Lambda_{\g_1}^{\natural}-\Lambda_{\g_2}^{\natural}\rVert \Big]\norm{b(y,\cdot)}_{H^2(S^+_y\M_{1})}.
\end{multline*}
This completes the proof of the lemma.
\end{proof}
\subsection{Proof of main result}
Let us now prove Theorem \ref{Th1}. We denote $\mathfrak{t}(x)=(\mathrm{t}_{jk})=\alpha(x)\mathfrak{s}(x)$ the symmetric $2$-tensor given by
\begin{equation}\label{4.16}
\mathrm{t}_{jk}(x)=\alpha(x)(\g_2^{jk}-\g_1^{jk})=\alpha(x)\mathrm{s}_{jk}(x),\quad j,k=1,\ldots,n.
\end{equation}
Since $\g_1^{jn}=\g_2^{jn}=\delta_{jn}$, $j=1,\ldots,n$, we obtain
\begin{equation}\label{4.17}
\mathrm{t}_{jn}=\mathrm{t}_{nj}=0,\quad j=1,\ldots,n.
\end{equation}
Furthermore we have
\begin{equation}\label{4.18}
\norm{\mathfrak{t}}_{\mathcal{C}^2}\leq C\norm{\g_2^{-1}-\g_1^{-1}}_{\mathcal{C}^2}=C\norm{\g_2^{-1}(\g_1-\g_2)\g_1^{-1}}_{\mathcal{C}^2}\leq C\varepsilon.
\end{equation}
Moreover the tensor $\mathfrak{t}=(\mathrm{t}_{jk})$ admits a decomposition into solenoidal and potential parts
\begin{equation}\label{4.19}
\mathfrak{t}=\mathfrak{t}^{\textrm{sol}}+\nabla_{\textrm{sym}}\mathbf{v},\quad \mathbf{v}=0,\,\,\textrm{on}\,\p\M,
\end{equation}
where $\mathfrak{t}^{\textrm{sol}}$ is the solenoidal part, $\mathbf{v}=(\mathrm{v}_1,\ldots,\mathrm{v}_n)$ is vector field in $\M$ and 
$$
\nabla_{\textrm{sym}}\mathbf{v}=\frac{1}{2}(\nabla_j\mathrm{v}_k+\nabla_k\mathrm{v}_j),\quad 1\leq j,k\leq n,
$$
$\nabla_j$ is the covariant derivative in metric $\g_1$. Note that a tensor $\mathfrak{t}^{\textrm{sol}}$ is called solenoidal if $\delta^s\mathfrak{t}^{\textrm{sol}}=0$, where $\delta^s$ is dual operator to $\nabla_{\textrm{sym}}$. 
\begin{lemma}\label{L4.3}
There exist $C>0$, $\mu_j>$,  $j=1,2,3,4$ such that the following estimate holds true
\begin{equation}\label{4.20}
\norm{\mathfrak{t}^{\textrm{sol}}}_{H^2(\M)}\leq C\Big(\varepsilon^{\mu_1}\norm{\mathfrak{t}}_{L^2(\M)}+ \gamma^{-\mu_2}+e^{\mu_3 \gamma}\| \Lambda_{\g_1}^{\natural}-\Lambda_{\g_2}^{\natural}\rVert^{\mu_4}\Big).
\end{equation}
\end{lemma}
\begin{proof}
We choose $b(y,\theta)=\I_{\g_1}\para{N_{\g_1}(\mathfrak{t})}(y,\theta)$ and obtain using Lemma \ref{L4.2} and \eqref{2.13}
\begin{multline*}
\norm{N_{\g_1}(\mathfrak{t})}^2_{L^2(\M_1)} \leq C\Big[\norm{\mathfrak{t}}_{\mathcal{C}^2}
\big(h+ h^{-3}\norm{\mathfrak{t}}_{\mathcal{C}^2}\big)\cr
+ \gamma^{-\frac{1}{2}}(h+ h^{-2}\norm{\mathfrak{t}}_{\mathcal{C}^2})+h^{-3}e^{\mu \gamma} \| \Lambda_{\g_1}^{\natural}-\Lambda_{\g_2}^{\natural}\rVert \Big]\norm{N_{\g_1}(\mathfrak{t})}_{H^2(\M_1)}.
\end{multline*}
By interpolation inequality, we have
\begin{multline*}
\norm{N_{\g_1}(\mathfrak{t})}^2_{H^2(\M_1)}\leq C \norm{N_{\g_1}(\mathfrak{t})}_{L^2(\M_1)}^{2/3}\norm{N_{\g_1}(\mathfrak{t})}^{4/3}_{H^3(\M_1)}
\leq C\Big[\norm{\mathfrak{t}}_{\mathcal{C}^2}
\para{h+h^{-3}\norm{\mathfrak{t}}_{\mathcal{C}^2}}\cr
+ \gamma^{-\frac{1}{2}}(h+h^{-2}\norm{\mathfrak{t}}_{\mathcal{C}^2})+h^{-3}e^{\mu \gamma} \| \Lambda_{\g_1}^{\natural}-\Lambda_{\g_2}^{\natural}\rVert \Big]^{1/3}\norm{N_{\g_1}(\mathfrak{t})}_{H^3(\M_1)}^{5/3}.
\end{multline*}
We use \eqref{2.15} and \eqref{2.14} to deduce
\begin{multline*}
\|\mathfrak{t}^{\textrm{sol}}\|^2_{L^2(\M)} \leq C\Big[\norm{\mathfrak{t}}_{\mathcal{C}^2}
\para{h+h^{-3}\norm{\mathfrak{t}}_{\mathcal{C}^2}}
+ \gamma^{-\frac{1}{2}}(h+h^{-2}\norm{\mathfrak{t}}_{\mathcal{C}^2})+h^{-3}e^{\mu \gamma} \| \Lambda_{\g_1}^{\natural}-\Lambda_{\g_2}^{\natural}\rVert \Big]^{\frac{1}{3}}\norm{\mathfrak{t}}_{\mathcal{C}^2}^{\frac{5}{3}} \\
\leq C\Big(\norm{\mathfrak{t}}^2_{\mathcal{C}^2}
\para{h^{1/3}+h^{-1}\norm{\mathfrak{t}}_{\mathcal{C}^2}^{1/3}}
+ \gamma^{-\frac{1}{6}}(h^{1/3}+h^{-2/3}\norm{\mathfrak{t}}_{\mathcal{C}^2}^{1/3})+h^{-1}e^{\mu \gamma} \norm{\mathfrak{t}}_{\mathcal{C}^2}^{5/3}\| \Lambda_{\g_1}^{\natural}-\Lambda_{\g_2}^{\natural}\rVert^{1/3} \Big).
\end{multline*}
Fixing $h=\norm{\mathfrak{t}}_{\mathcal{C}^2(\M)}^{1/4}$,  we get
\begin{align*}
\norm{\mathfrak{t}^{\textrm{sol}}}^2_{L^2(\M)} &\leq C \Big(\norm{\mathfrak{t}}_{\mathcal{C}^2(\M)}^{2(1+\alpha)}+ \gamma^{-\frac{1}{6}}+e^{\mu \gamma}\| \Lambda_{\g_1}^{\natural}-\Lambda_{\g_2}^{\natural}\rVert^{1/3}\Big) \\
&\leq C \Big(\varepsilon^{\alpha} \norm{\mathfrak{t}}_{\mathcal{C}^2(\M)}^{2+\alpha}+\gamma^{-\frac{1}{6}}+e^{\mu \gamma}\| \Lambda_{\g_1}^{\natural}-\Lambda_{\g_2}^{\natural}\rVert^{1/3}\Big), 
\end{align*}
for some $\alpha>0$. For $\beta_1=\frac{2+\alpha/2}{2+\alpha}$, we can find $k_1$ sufficiently large such that
\begin{align*}
\norm{\mathfrak{t}}_{\mathcal{C}^2(\M)}&\leq C \norm{\mathfrak{t}}_{H^{n/2+3}(\M)} \\  
&\leq C \norm{\mathfrak{t}}_{L^2(\M)}^{\beta_1}\norm{\mathfrak{t}}_{H^{k_1}(\M)}^{1-\beta_1} \leq C \norm{\mathfrak{t}}_{L^2(\M)}^{\beta}
\end{align*}
we therefore obtain
$$
\norm{\mathfrak{t}^{\textrm{sol}}}_{L^2(\M)}\leq C \varepsilon^{\alpha/2} \norm{\mathfrak{t}}_{L^2(\M)}^{1+\frac{\alpha}{4}}+ \Big(\gamma^{-\frac{1}{12}}+e^{\mu \gamma}\| \Lambda_{\g_1}^{\natural}-\Lambda_{\g_2}^{\natural}\rVert^{1/6}\Big)
$$
Moreover, for $\beta_2=\frac{1}{1+\alpha/4}$, we can find $k_2$ sufficiently large we have
\begin{equation}\label{4.21}
\norm{\mathfrak{t}^{\textrm{sol}}}_{H^2(\M)}\leq \norm{\mathfrak{t}^{\textrm{sol}}}^{\beta_2}_{L^2(\M)}\norm{\mathfrak{t}^{\textrm{sol}}}^{1-\beta_2}_{H^{k_2}(\M)}\leq C\Big(\varepsilon^{\mu_1}\norm{\mathfrak{t}}_{L^2(\M)}+ \gamma^{-\mu_2}+e^{\mu_3 \gamma}\| \Lambda_{\g_1}^{\natural}-\Lambda_{\g_2}^{\natural}\rVert^{\mu_4}\Big).
\end{equation}
This completes the proof.
\end{proof}

\begin{lemma}\label{L4.4} 
Let $\mathfrak{t}$ be given by (\ref{4.16}). Then there exist a constant $C>0$ such that, the following estimate holds true
\begin{equation}\label{4.22}
\norm{\mathfrak{t}}_{L^2(\M)}\leq C\norm{\mathfrak{t}^{\textrm{sol}}}_{H^2(\M)}.
\end{equation}
\end{lemma}
\begin{proof}
 Since $\mathfrak{t}=(\mathrm{t}_{jk})_{jk}$ is such that $\mathrm{t}_{nk}=0$ in $\M$ we can estimate
\begin{equation}\label{4.23}
\norm{\nabla_{\textrm{sym}}\mathbf{v}}_{L^2(\M)}\leq C\norm{\mathfrak{t}^{\textrm{sol}}}_{H^{2}(\M)}.
\end{equation}
Indeed we have
\begin{equation}\label{4.24}
\nabla_n\mathrm{v}_k+\nabla_k\mathrm{v}_n=-2\mathrm{t}^{\textrm{sol}}_{kn},\quad k=1,\ldots,n.
\end{equation}
In particular $\nabla_n\mathrm{v}_n=-\mathrm{t}^{\textrm{sol}}_{nn}$, and we have
\begin{equation}\label{4.25}
\nabla_n\mathrm{v}_j=\frac{\p \mathrm{v}_j}{\p x_n}-\sum_{k=1}^n\Gamma_{jn}^k\mathrm{v}_k=\frac{\p \mathrm{v}_j}{\p x_n}-\frac{1}{2}\sum_{\ell, k=1}^n\g_1^{k\ell }\frac{\p\g_{1\ell k}}{\p x_n}\mathrm{v}_k,
\end{equation}
where we have used
$$
\Gamma_{jn}^k=\frac{1}{2}\sum_{\ell=1}\g_1^{k\ell}\para{\frac{\p \g_{1\ell j}}{\p x_n}+\frac{\p \g_{1\ell n}}{\p x_j}-\frac{\p \g_{1n j}}{\p x_\ell} }= \frac{1}{2}\sum_{\ell=1}\g_1^{k\ell}\frac{\p \g_{1\ell j}}{\p x_n}.
$$
Therefore, since $\g_{1\ell n}=\delta_{\ell n}$, we get 
$$
\nabla_n\mathrm{v}_n=\frac{\p \mathrm{v}_n}{\p x_n},
$$
Thus
\begin{equation}\label{4.26}
\mathrm{v}_n=-\int_{-\infty}^{x_n} \mathrm{t}_{nn}^{\textrm{sol}}(x)dx_n,
\end{equation}
and
\begin{equation}\label{4.27}
\norm{\mathrm{v}_n}_{H^2(\M)}\leq C\norm{\mathrm{t}^{\textrm{sol}}_{nn}}_{H^2(\M)}.
\end{equation}
Then we can estimate $\mathrm{v}_j$, $j=1,\ldots,n-1$ from the relation
\begin{equation}\label{4.28}
\frac{\p \mathrm{v}_j}{\p x_n}-\frac{1}{2}\sum_{\ell, k=1}^n\g_1^{k\ell }\frac{\p\g_{1\ell k}}{\p x_n}\mathrm{v}_k=-2\mathrm{t}^{\textrm{sol}}_{jn}-\nabla_j\mathrm{v}_n.
\end{equation}
Integrating in $x_n$, we get for $j=1,\ldots,n-1$, 
\begin{equation}\label{4.29}
\abs{\mathrm{v}_j(x',x_n)}\leq C\para{\sum_{k=1}^{n-1}\int_{-R}^{x_n}\abs{\mathrm{v}_k(x',y_n)}dy_n+\int_{-R}^{x_n}\para{\abs{\mathrm{t}^{\textrm{sol}}_{jn}(x',y_n)}+\abs{\mathrm{v}_n(x',y_n)}+ \abs{\nabla_j\mathrm{v}_n(x',y_n)}}dy_n}.
\end{equation}
We denote
\begin{equation*}\label{4.30}
\phi(x_n)=\sum_{j=1}^{n-1}\abs{\mathrm{v}_j(x',y_n)},\quad \zeta(x_n)=\sum_{j=1}^{n-1}\int_{-R}^{x_n}\para{\abs{\mathrm{t}^{\textrm{sol}}_{jn}(x',y_n)}+\abs{\mathrm{v}_n(x',y_n)}+ \abs{\nabla_j\mathrm{v}_n(x',y_n)}}dy_n,
\end{equation*}
we deduce that
\begin{equation}\label{4.31}
\phi(x_n)\leq C\para{\int_{-R}^{x_n}\phi(y_n)dy_n+\zeta(x_n)}.
\end{equation}
Applying Gr\"onwall's inequality we get
\begin{equation*}\label{4.32}
\abs{\mathrm{v}_j(x)}\leq \phi(x_n)\leq C\para{\zeta(x_n)+\int_{-R}^{x_n}\zeta(y_n)dy_n}, \quad j=1,\ldots,n-1
\end{equation*}
Integrating with respect $x\in\M$, we find
\begin{equation}\label{4.33}
\norm{\mathrm{v}_j}_{L^2(\M)}\leq C\sum_{j=1}^{n-1}\norm{\mathrm{t}^{\textrm{sol}}_{jn}}_{L^2(\M)}+C\norm{\mathrm{v}_n}_{H^1(\M)}\leq C\norm{\mathfrak{t}^{\textrm{sol}}}_{H^2(\M)}.
\end{equation}
Similarly, taking $\nabla_m$, $m=1,\ldots,n$ for (\ref{4.28}) and applying  Gr\"onwall's inequality, we get 
\begin{equation}\label{4.34}
\norm{\mathrm{v}_j}_{H^1(\M)}\leq C\sum_{j=1}^{n-1}\norm{\mathrm{t}^{\textrm{sol}}_{jn}}_{H^1(\M)}+C\norm{\mathrm{v}_n}_{H^2(\M)}\leq C\norm{\mathfrak{t}^{\textrm{sol}}}_{H^2(\M)}.
\end{equation}
Therefore
\begin{equation}\label{4.35}
\norm{\nabla_{\textrm{sym}}\mathrm{v} }_{L^2(\M)}\leq C\norm{\mathfrak{t}^{\textrm{sol}}}_{H^2(\M)}.
\end{equation}
Taking $\mathfrak{t}=\mathfrak{t}^{\textrm{sol}}+\nabla_{\textrm{sym}}\mathrm{v}$, we get
\begin{equation}\label{4.36}
\norm{\mathfrak{t}}_{L^2(\M)}\leq C\norm{\mathfrak{t}^{\textrm{sol}}}_{H^2(\M)}.
\end{equation}
This completes the proof.
\end{proof}
We return now to the proof of the main result. By (\ref{4.20}) and (\ref{4.22}) and taking $\varepsilon$ sufficiently small we get
\begin{equation}\label{4.37}
\norm{\mathfrak{s}}_{L^2(\M)}=\norm{\g^{-1}_2-\g^{-1}_1}_{L^2(\M)}\leq C\norm{\mathfrak{t}}_{L^2(\M)}\leq C\para{\gamma^{-\mu_2}+e^{\mu_3\gamma}\norm{\Lambda^\natural_{\g_1}-\Lambda^\natural_{\g_2}}^{\mu_4}}.
\end{equation}
Writing $\g_1-\g_2=\g_2\mathfrak{s}\g_1$, we obtain
\begin{equation}\label{4.38}
\norm{\g_1-\g_2}_{L^2(\M)}\leq C\norm{\mathfrak{s}}_{L^2(\M)}\leq  C\para{\gamma^{-\mu_2}+e^{\mu_3\gamma}\norm{\Lambda^\natural_{\g_1}-\Lambda^\natural_{\g_2}}^{\mu_4}}.
\end{equation}
Selecting
$$
\gamma=\frac{1}{2\mu_2}\log\para{2+\norm{\Lambda^\natural_{\g_1}-\Lambda^\natural_{\g_2}}^{-\mu_4}}
$$ 
we obtain (\ref{2.42}) and then (\ref{1.18}).\\
This ends the proof of theorem.



\end{document}